\documentclass[preprint,3p]{elsarticle}

\usepackage{amssymb}
\usepackage{amsmath}

\usepackage{todonotes}
\usepackage{caption}
\usepackage{subcaption}
\usepackage{makecell}
\usepackage{booktabs}
\usepackage{tabularray}
\UseTblrLibrary{booktabs}
\usepackage{comment}

\usepackage{amsthm}
\usepackage{algorithm}
\usepackage[noend]{algpseudocode}

\providecommand{\norm}[1]{\left\lVert#1\right\rVert}

\newcommand{\weak}{\rightharpoonup}

\newcommand{\be}{\begin{equation}}
\newcommand{\ee}{\end{equation}}
\newcommand{\ba}{\begin{eqnarray}}
\newcommand{\ea}{\end{eqnarray}}
\newcommand{\beq}{\begin{equation}}
\newcommand{\eeq}{\end{equation}}

\renewcommand{\leq}{\leqslant}
\renewcommand{\le}{\leqslant}
\renewcommand{\geq}{\geqslant}

\def \R {\mathbb{R}}

\def \dis {\displaystyle}

\def\beq{\begin{equation}}
\def\eeq{\end{equation}}
\def\ecart{\noalign{\medskip}}
\def\ba{\begin{array}}
\def\ea{\end{array}}

\DeclareMathOperator*{\argmin}{argmin}
\DeclareMathOperator*{\argmax}{argmax}

\usepackage[nameinlink]{cleveref}
\numberwithin{equation}{section}
\newtheorem{theo}{Theorem}[section]
\newtheorem{prop}{Proposition}[section]
\newtheorem{cor}{Corollary} [section]
\newtheorem{lemma}{Lemma}[section]
\newtheorem{defi}{Definition}[section]
\newtheorem{hypo}{Hypothesis}[section]

\newtheorem{rem}[defi]{Remark}

\journal{Journal of Differential Equations}

\begin{document}

\begin{frontmatter}



\title{Smooth approximation of feedback laws for infinite horizon control problems with non-smooth value functions}
\author[RICAM,Graz]{Karl Kunisch}
\author[RICAM]{Donato Vásquez-Varas}

\affiliation[RICAM]{organization={Radon Institute for Computational and Applied Mathematics, Austrian Academy of Sciences},
             addressline={Altenbergerstraße 69},
             city={Linz},
             postcode={4040},
             state={Upper Austria},
             country={Austria}}
\affiliation[Graz]{organization={Institute of Mathematics and Scientific Computing, University of Graz},
             addressline={Heinrichstraße 36},
             city={Graz},
             postcode={8010},
             state={Styria},
             country={Austria}}


\begin{abstract}
 In this work the synthesis of  approximate optimal and smooth feedback laws for infinite horizon optimal control problems is addressed. In this regards,  $L^{p}$ type error bounds of the approximating smooth feedback laws are derived, depending on either the $C^1$ norm of the value function or its semi-concavity. These error bounds combined with the existence of a Lyapunov type function are used to prove the existence of an approximate optimal sequence of smooth feedback laws. Moreover, we extend this result to the Hölder continuous case by a diagonalization argument combined with the Moreau envelope. It is foreseen that these error bounds could be applied to study the  convergence of synthesis of feedback laws via data driven machine learning methods. Additionally, we provide an example of an infinite horizon optimal control problem for which the value functions is non-differentiable but Lipschitz continuous. We point out that in this example no restrictions on either the controls or the trajectories are assumed.
\end{abstract}



\end{frontmatter}


\section{Introduction}
The construction of feedback laws for infinite horizon control problems is a hard task. The classical approach depends on  solving the Hamilton Jacobi Bellman equation (HJB). Among the wellknown methods for solving the HJB equation we mention finite difference schemes \cite{Bonnans}, semi-Lagrangian schemes \cite{Falcone}, and policy iteration \cite{Alla, Beard, Puterman, Santos}. However, it is known that in general these approaches suffer from the \emph{curse of dimensionality}. In the last years  significant  efforts have been dedicated to overcome this difficulty by using machine learning and deep learning techniques. For instance we can mention the following contributions:  representation formulas \citep{Chow1,Chow2,Chow3,DarbonOsher}, approximating the HJB equation by neural networks \citep{Han,Darbon,Nusken,Onken,Ito,KW1,Ruthotto,chen2023deep,zhao2023offline}, data driven approaches \citep{Nakamura1,Nakamura2,AzKaKK,Kang,Albi,ehring}, max-plus methods \citep{Akian,Gaubert,Dower}, polynomial approximation \citep{Kalise1,Kalise2}, tensor decomposition
methods \citep{Horowitz,Stefansson,Gorodetsky,Dolgov,Oster,Oster2}, POD methods \citep{Alla2,KunischVolk}, tree structure algorithms \citep{Alla1}, and sparse grids techniques\citep{BokanowskiGarckeGriebelPo, Garcke, KangWilcox,BokaWarinProst}, see also the proceedings volume \citep{KaliseKuRa}. In general, there is no  proof concerning the fact that feedback laws constructed via data driven approaches or the resolution of HJB by neural networks provide optimal controls, unless the value function some smoothness of the value function is required as is the case in \cite{ehring}, where the value functions is supposed to be an element of a reproducing kernel Hilbert space. There is a group of techniques, based  on training feedback laws involving the value function, by  minimizing the average of the objective function of the control problem over a set of initial conditions \cite{KuVaWal,KW1,KW2,KuVa,BokaWarinProst}. For these approaches, convergence results have been provided under the assumption that the value function of the control problem is $C^{1,1}$ and the optimal states are bounded. However, these assumptions can be weakened by assuming the existence of a sequence of smooth feedback's converging  in such a manner  that the states generated by these feedback laws remain bounded. In this work, we provide conditions which ensure the existence of such sequences of feedback laws. We distinguish four cases depending on whether the value function is $C^{1}$, semi-convex, semi-concave, or $\alpha$-Hölder continuous with $\alpha\in \left(\frac{1}{2},1\right]$. For each of them we also assume the existence of an appropriately chosen  Lyapunov function. For this purpose  we derived $L_p$ type error bounds with respect to the optimality of smooth feedback laws which could also be used to study the convergence of feedback laws provided by data driven approaches.

The structure of this work is as follows. In \Cref{Sec:ControlTheory} we present the control problem and some important concept of the dynamical programming approach for synthesizing feedback laws. In \Cref{Sec:PrelResult} we introduce  preliminary results on semi-concave functions, viscosity solutions, and their regularization by means of the  Moreau envelope. In \Cref{Sec:DefHyp} we introduce  necessary definitions and we present the main assumption about the stabilizability   by feedback control.
  In \Cref{MainResults} we state our main results which are the construction of a sequence of feedback laws such that the state given by this approximation is bounded and the evaluation of the objective function in the obtained controls converges to the value function. For the proofs we refer to later sections, mainly to section \ref{Convergence result}.
\Cref{Sec:ErrorEs} is devoted to obtaining estimates on the error between the value function and the evaluation of the objective function of the control problem, assuming that the state is bounded. In \Cref{Sec:StabilityEst}, a result for the stability of the trajectories of approximations of feedback laws constructed by using an approximation of the value function is provided.   The final section is devoted to the detailed description of an example for an optimal control problem which admits at least two globally optimal solutions for initial conditions in an appropriately chosen subset of the state space. As a consequence it will follow that the value function is not $C^1$ on this set of initial conditions. However, all the assumption for the applicability of the presented results can still be verified.

\section{Control theory and dynamic programming}
\label{Sec:ControlTheory}
We consider the following control problem
\beq
 (P) \quad\min_{u} J(y_0,u):=\int_{0}^{\infty}\left(\ell(y)+\frac{\beta}{2}|u|^{2}\right) dt\label{ControlProblem}\eeq
\beq s.t.\quad \dot{y}(t)=f(y(t))+B(y(t))u(t),\quad y(0)=y_0,  \label{OpenLoop}
\eeq
where $T>0$, $\ell \in Lip_{loc}(\R^{d})$ is bounded from below by 0, with $\ell(0)=0$, $y_0\in \R^{d}$, $\beta>0$, $u\in L^2((0,\infty); \R^m)$, $f\in Lip_{loc}(\R^{d};\R^{d})$ with $f(0)=0$, and $B\in Lip_{loc}(\R^{d};\R^{d\times m})$. We define the associated value function by
\beq V(y_0)=\min_{u\in L^{2}((0,\infty);\R^{m})} J(y_0,u).\eeq
It is known that $V$ satisfies the dynamic programming principle, i.e., for all $T>0$:
\beq
V(y_0)=\min_{u\in L^{2}((0,T);\R^{m})} \left\{\int_{0}^{T}\left(\ell(y(\cdot;y_0))+\frac{\beta}{2}|u|^2\right)dt+V(y(T;y_0))\right\},
\label{DynamicProgrammic}
\eeq
where $y(\cdot;y_0)$ is the solution of \eqref{OpenLoop}. Further, if problem \eqref{ControlProblem} has a solution for all $y_0\in \Omega$ with $\Omega\subset \R^d$ open, then $V$ satisfies the Hamilton-Jacobi-Bellman equation $\mathcal{H}(y,\nabla v(y))=0$ in the viscosity sense,
where
$$\mathcal{H}(y,p)= \max_{u\in \R^{M}} H(y,p,u),
$$
and
\beq H(y,p,u)= \left\{-p\cdot (f(y)+B(y)u)-\ell(y)-\frac{\beta}{2}|u|^2\right\}. \label{Hamiltonian}\eeq
Thus, in the case that the cost is quadratic  in $u$ and the control enters into \eqref{OpenLoop} in an affine manner,  the HJB equation becomes
\beq
\mathcal{H}(y,\nabla v(y))=-\ell(y)+\frac{1}{2\beta}|B^{\top}(y)\nabla v(y)|^2-\nabla v(y)^{\top}f(y)=0\mbox{ in }\Omega. \label{HJBeq}\eeq
The following properties of $H$  and $\mathcal{H}$ will be important in the upcoming sections:
 \beq
 \begin{array}{l}
      \dis \left|H(y,p_1,u_1)-H(y,p_2,u_2)\right|\leq \frac{\beta}{2}|u_1+u_2|\cdot|u_1-u_2|+|p_1|\cdot|B(y)|\cdot|u_1-u_2|+|f(y)|\cdot |p_1-p_2| \\
      \ecart\dis +|u_2|\cdot |B(y)|\cdot |p_1-p_2|,
 \end{array}
 \label{Hamiltonian:prop}\eeq
  \beq
 \begin{array}{l}
      \dis \left|\mathcal{H}(x,p_1)-\mathcal{H}(y,p_1)\right|\leq \left(\norm{\ell}_{Lip(\Omega)}+\frac{1}{\beta}\norm{B}_{Lip(\Omega;\R^{m\times d})}^2|p_1|^2+\norm{f}_{Lip(\Omega;\R^{d})}|p_1|\right)|x-y|,
 \end{array}
 \label{Hamiltonian:prop:2}\eeq

 for all $u_1,u_2\in\R^{M}$, $ p_1,p_2\in\R^{d}$, and $x,y\in \Omega$, where we use the following definition for $\norm{B}_{Lip(\Omega;\R^{m\times d})}$
 \beq \norm{B}_{Lip(\Omega;\R^{m\times d})}=\sup_{\begin{array}{c}
      x,y\in\Omega,\\ i\in 1,\ldots m,\\ j\in 1,\ldots, d
 \end{array}}\frac{|B_{i,j}(x)-B_{i,j}(y)|}{|x-y|}+|B_{i,j}(x)-B_{i,j}(y)|\eeq

 Additionally, if $V$ is smooth, then $V$ is a classical solution of the  HJB equation and this allows us to construct a feedback law. Namely, if for an initial condition $y_0\in \Omega$ the optimal solution of \eqref{ControlProblem} is such that the respective state $y^{*}$satisfies $y^{*}(t)\in\Omega$ for all $t\in(0,\infty)$, then we can define a feedback law by
\beq
u^{*}(y)= \argmax_{u\in R^{M}} H(y,\nabla V(y),u)\label{OptimalFeedbackEq}
\eeq
along $y=y^*(t)$,
where $y^{*}$ solves the closed loop problem
\beq \frac{d}{dt}y^{*}(t)=f(y^{*}(t))+B(y^{*}(t))u^{*}(y^{*}(t)),\quad y^{*}(t_0)=y_0. \label{closedloop1}\eeq
This implies that once we have solved the HJB equation, we can obtain an optimal feedback law by using \eqref{OptimalFeedbackEq}. However, in general the value function is not differentiable.
 But we will see in the following sections that it is possible to construct a feedback such that $J$ evaluated along the controls provided by this feedback is close to the value function.
We will explain this in more detail in \Cref{Sec:DefHyp}. First we need some preliminary results which are given in the next section. Throughout $V$ denotes the value function of \eqref{ControlProblem}. We assume that it is a continuous function in $\R^d$.

\section{Preliminary results}
\label{Sec:PrelResult}
In this section we provide  results that we will need throughout the article. We start by recalling the definition of a viscosity solution and some results concerning its regularity. Then, we present a definition of semi-concave functions together with some of their useful properties. All the results and definitions in this section can be found in \cite[Chapters 1 and 2]{FlemSon}, \cite[Chapters~1, 2 and 3]{Cannarsa2004} and \cite[Chapter~2]{Bardi1997}, except for those concerning the Moreau envelope of a Hölder function. Throughout  $\Omega\subset\R^{d}$ denotes  an open.

\begin{defi}[Viscosity solution, {see \cite[Chapter 2, Definition 1.1]{Bardi1997}}]
\label{ViscositySol}
Let $F\in C(\Omega\times\R^d)$. We say that $v\in C(\Omega)$ is a sub-solution of
\beq F(y,\nabla v(y))=0 \mbox{ in }\Omega \label{ViscositySol:eq1}\eeq
in the viscosity sense if for all $\bar{y}\in \Omega$  and $\phi\in C_{loc}^{1}(\Omega)$ such that $v-\phi$ attains a local maximum at $\bar{y}$ it holds that
\beq  F(\bar{y},\nabla \phi(\bar{y}))\leq 0. \label{ViscositySol:eq2}\eeq
Analogously, we say that $v\in C(\Omega)$ is super-solution of
\eqref{ViscositySol:eq1} in the viscosity sense if for all $\bar{y}\in \Omega$  and $\phi\in C_{loc}^{1}(\Omega)$ such that $v-\phi$ attains a local minimum at $\bar{y}$ it holds that
\beq F(\bar{y},\nabla \phi(\bar{y}))\geq 0. \eeq
We say that $v$ is a viscosity of \eqref{ViscositySol:eq1} if it is a sub and a super solution of \eqref{ViscositySol:eq1}.
\end{defi}

\begin{prop}[{see \cite[Chapter 2, Proposition 1.9]{Bardi1997}}]
    Let $F\in C(\Omega\times\R^d)$. If $v$ is a viscosity solution of \eqref{ViscositySol:eq1}
    which is differentiable at $\bar{y}$, then $F(\bar y,\nabla v(\bar y))=0$. Further, if $v$ is locally Lipschitz in $\Omega$, then
     $$ F(y,\nabla v(y))=0 \mbox{ a.e. in }\Omega.$$
\end{prop}
We now give a definition of a semi-concave functions which is suitable for what we are intending. However, it is possible to find more general definitions in the literature.
\begin{defi}[Semi-concave function, {see \cite[Chapter 2, Section 4.2]{Bardi1997}}]
    Let $\Omega$ be convex. A function $v\in C(\Omega)$ is semi-concave in $\Omega$, if there exists a constant $C>0$ such that $x\in\Omega\mapsto v(x)-\frac{C}{2}|x|^{2}$ is concave. If $-v$ is semi-concave we say that $v$ is semi-convex.
\end{defi}

There are several interesting properties of semi-concave functions, however, we only need the following:
\begin{prop}
\label{Simconvexity:c11}
    Let $\Omega$ be convex and $v$ be semi-concave in $\Omega$. Then,
    \begin{enumerate}
        \item (see \cite[Chapter 2, Theorem 2.3.1]{Cannarsa2004}) $v$ is locally Lipschitz continuous in $\Omega$ and $\nabla v\in BV_{loc}(\overline{\Omega})^{d}$.
        \item (see  \cite[Chapter 3, Theorem 3.3.7]{Cannarsa2004}) If additionally $v$ is semi-convex, then $v\in C^{1,1}(\Omega)$.
    \end{enumerate}
\end{prop}
It will be necessary to regularize semi-concave function. For this purpose, it is useful to recall the definition of a smooth mollifier and the mollification of a function.
\begin{defi}
    A compactly supported function $\rho:\R^{d}\to \R_{+}$ is a smooth mollifier if it is infinite differentiable and it satisfies
$$\int_{\R^{d}}\rho(x)dx=1,\quad \lim_{\varepsilon\to 0^{+}}\int_{\R^{d}}\phi(x)\rho(x/\varepsilon)\varepsilon^{-d}dx=\phi(0) \mbox{ for all }\phi\in C(\R^{d}).$$
For $\varepsilon>0$ and $\phi\in L_{loc}^{1}(\R^{d})$ we say that $\phi_{\varepsilon}=\phi\ast\rho_{\varepsilon}$ is a mollification of $\phi$, where $\rho_{\varepsilon}(x)=\frac{1}{\varepsilon}\rho(x/\varepsilon)$ and $\rho$ a smooth mollifier.
\end{defi}
The following  property of a mollification of a semi-concave function is a consequence of Proposition 1.3.3 in \cite{Cannarsa2004}:
\begin{prop}
\label{prop:hessbound}
    Let $\Omega$ be convex and bounded, and $v\in C(\Omega)$ be a semi-concave function with constant $C>0$. Let $v_{\varepsilon}$ be a mollification of $v$ and set $\Omega_{\varepsilon}=\{x\in\Omega: dist(x,\partial\Omega)> \varepsilon\}$. Then $v_{\varepsilon}$ is semi-concave with constant $C>0$ in $\Omega_{\varepsilon}$ and  it satisfies
    $$ p^{\top}\nabla^{2} v_{\varepsilon}(x)p\leq C|p|^2\mbox{ for all }x\in \Omega_{\varepsilon}\mbox{ and for all }p\in\R^{d}.$$
\end{prop}

We will derive  error bounds depending  on regularity and structural properties of the value function. We start with the case that $V$ is $C^1$ regular. In the case that the value function is not $C^1$ but semi-convex,
we shall employ a Moreau approximation which is  $C^{1,1}$ regular, so that the
the $C^1$ error bound can be used.
In the semi-concave case we shall make use a mollification, which in view of  Proposition \ref{prop:hessbound} preserves the semi-concavity. For the non semi-concave case  we shall use again mollification  combined with the Moreau envelope.
 In this  way we shall be able to control the blow-up of the semi-concavity constant.

Let us start by recalling  the definition of the Moreau envelope.

\begin{defi}
    \label{MoreauEnv}
     For $\phi\in C(\Omega)\cap L^{\infty}(\Omega)$ and $\lambda>0$ the Moreau envelope of $\phi$ is defined by
    $$ \mathcal{M}_{\lambda}\phi(y)=\inf_{x\in\Omega} \phi(x)+\frac{1}{2\lambda}|x-y|^{2}.$$
\end{defi}
The next proposition summarizes some important features of $\mathcal{M}_{\lambda}$ for $\phi$ bounded and continuous, including that $\mathcal{M}_{\lambda}\phi$ is semi-concave. The results can be found in  \cite{Bardi1997} except for   \eqref{MoreauEnvProp:item0:2} which is contained  in  \cite{Lasry1986}, and \eqref{MoreauEnvProp:item00:1}, which is verified below.
\begin{prop}
\label{MoreauEnvProp}
 Let $\phi\in C(\Omega)\cap L^{\infty}(\Omega)$. For $\lambda >0$ we denote
 $$\Omega^{\lambda}_{\phi}=\left\{x\in \Omega: dist(x,\partial\Omega)\geq 2\lambda^{\frac{1}{2}}\norm{\phi}_{L^{\infty}(\Omega)}^{\frac{1}{2}} \right\}$$
 and for $x\in \Omega$
 $$ \mathcal{A}_{\lambda}\phi(x)=\argmin_{y\in \Omega} \phi (y)+\frac{1}{2\lambda}|x-y|^2.$$
 Then the following statements hold:
\begin{enumerate}[(a)]
    \item (see  \cite[Chapter 2, Lemma 4.11 and Lemma 4.12]{Bardi1997})\label{MoreauEnvProp:item00} For all $x\in \Omega^{\lambda}_{\phi}$ we have $\mathcal{A}_{\lambda}\phi(x)\neq \emptyset$  and for all $y\in \mathcal{A}_{\lambda}\phi(x)$ we have
    \beq |x-y|\leq 2\lambda^{\frac{1}{2}}\norm{\phi}_{L^{\infty}(\Omega)} \label{MoreauEnvProp:item00:0}\eeq
    In addition, if $\phi$ is $\alpha$-Hölder continuous in $\overline{\Omega}$, i.e., there exist $\alpha\in (0,1]$ and $C>0$ such that
    $$ |\phi(x_1)-\phi(x_2)|\leq C|x_1-x_2|^{\alpha}\mbox{ for all }x_{1},x_2\in \overline{\Omega},$$
    then
    \beq |x-y|\leq (2C\lambda)^{\frac{1}{2-\alpha}}.\label{MoreauEnvProp:item00:1}\eeq
    \item (see  \cite[Chapter 2, Lemma 4.12]{Bardi1997})\label{MoreauEnvProp:item0} The function $\mathcal{M}_{\lambda}\phi$ converges to $\phi$ in $C_{loc}(\Omega)$ as $\lambda\to 0^{+}$ and
    \beq\lim_{\lambda \to 0^{+}}\sup_{y\in \mathcal{A}_{\lambda }(x)}\frac{|x-y|^{2}}{2\lambda }=0\mbox{ in }C_{loc}(\Omega).\label{MoreauEnvProp:item0:1}\eeq
    Moreover, if  $\phi$ is $\alpha$-Hölder continuous we have
    \beq |\mathcal{M}_{\lambda}\phi(x)-\phi(x)|\leq C^{\frac{2}{2-\alpha}}2^{\frac{\alpha}{2-\alpha}}\lambda^{\frac{\alpha}{2-\alpha}} \mbox{ for all }x\in \Omega^{\lambda }\phi\label{MoreauEnvProp:item0:2}\eeq
    \item (see  \cite[Chapter 2, Lemma 4.11 and Lemma 4.12]{Bardi1997}) \label{MoreauEnvProp:item1} For every $\lambda>0$ the mapping $\mathcal{M}_{\lambda}\phi$ is  semi-concave  with constant $\frac{1}{\lambda}$ in $\Omega^{\lambda}_\phi$. In particular, for almost every $x\in \Omega^{\lambda}_{\phi}$ we have that $\mathcal{A}_{\lambda}\phi(x)$ is a singleton and
    $$ \nabla \mathcal{M}_{\lambda}\phi(x)=\frac{1}{\lambda}(x-y)\mbox{ where } \mathcal{A}_{\lambda}\phi(x)=\{y\}.$$
    \item (see  \cite[Chapter 2, Proposition 4.13]{Bardi1997})\label{MoreauEnvProp:item2} Consider $F\in C(\Omega\times\R^{d})$ such that for every $R>0$ there exists $h\in C([0,\infty); [0,\infty))$ satisfying $h(0)=0$ and
    $$ |F(y_{1},p)-F(y_{2},p)|\leq h(|y_{1}-y_{2}|\cdot(1+|p|))\mbox{ for all }y_{1},y_{2}\in \Omega\cap B(0,R),\ p\in \R^{d}.$$
    If $\phi$ is a viscosity super-solution of $F(y,\nabla \phi (y))=0$, then
    $$ F(y,\nabla \mathcal{M}_{\lambda}\phi (y))+g_{\lambda}(y)\geq 0 \mbox{ almost everywhere in }\Omega^{\lambda}_{\phi},$$
    where
    $$ g_{\lambda}(x)=\sup_{y\in \mathcal{A}_{\lambda}\phi (x)}  h\left(|x-y|\left(1+\frac{|x-y|}{\lambda}\right)\right).$$
    Further, it holds that
    $$ \lim_{\lambda\to 0^{+}}g_{\lambda}(x)=0\mbox{ in }C_{loc}(\Omega).$$
    \item (see \cite[Chapter 3, Theorem 3.5.3]{Cannarsa2004})\label{MoreauEnvProp:item3} If $\phi$ is semi-convex with constant $C>0$, then $\mathcal{M}_{\lambda}\phi$ is  C$^{1,1}(\Omega^{\lambda}_{\phi})$ and convex for all $\lambda\in \left(0,\frac{1}{C}\right)$.
\end{enumerate}
\end{prop}
As announced above, to verify \eqref{MoreauEnvProp:item00:1} let $x\in\Omega$ and $y\in \mathcal{A}_\lambda \phi(x)$, where we make use of the fact that $\mathcal{A}_\lambda \phi(x)$ is nonempty. Then
\begin{equation*}
\frac{1}{2\lambda}|x-y|^2 = \mathcal M_\lambda \phi(x) - \phi(y) \le \phi (x) - \phi (y) \le C|x-y|^\alpha,
\end{equation*}
from which the desired estimate follows.
The result in \Cref{MoreauEnvProp} (\ref{MoreauEnvProp:item2}) is slightly restrictive. For instance, this result can not be applied to a super-solution of \eqref{HJBeq}, unless we assume Lipschitz continuity of the value function. In the following result we extend this to non-Lipschitz functions by assuming Hölder continuity with an exponent large enough. It will used in the context
approximating the HJB equation as well as for the use of a Lyapunov function to approximate the escape time from the domain of the approximation of the HJB equation.

\begin{prop}
    \label{MoreauEnvProp:HolderConv}
    Let us assume that $\Omega$ is convex and let $F\in C(\Omega\times\R^{d})$ and $q\in [1,\infty)$ be such that for every $R>0$, there exists $h_i\in C([0,\infty);[0,\infty))$, satisfying $h_i(0)=0$, for $i\in \{1,2,3\}$,  and
    \beq  |F(y_1,p)-F(y_2,p)|\leq h_1(|y_1-y_2|)|p|^{q}+h_2(|y_1-y_2|)|p|+h_3(|y_1-y_2|))\mbox{ for all }y_1,y_2\in\Omega\cap B(0,R),\ p\in\R^{d}.\label{MoreauEnvProp:HolderConv:eq1}\eeq
    If $\phi$ is a viscosity super-solution of $F(y,\nabla \phi (y))=0$, then
    \beq F(y,\nabla \mathcal{M}_{\lambda}\phi (y))+g_{\lambda}(y)\geq 0 \mbox{ a.e. in }\Omega^{\lambda}_{\phi},\label{MoreauEnvProp:HolderConv:eq2}\eeq
    where
    \beq g_{\lambda}(x)=\sup_{y\in \mathcal{A}_{\lambda}\phi (x)}  h_1(|x-y|)\frac{|x-y|^{q}}{\lambda^q}+h_2(|x-y|)\frac{|x-y|}{\lambda}+h_3(|x-y|).\label{MoreauEnvProp:glambda}\eeq
    If additionally, $\phi$ is $\alpha-$Hölder continuous  with $\alpha\in \left(1-\frac{1}{q},1\right]$ and there exists $C>0$ such that
    \beq \ h_i(s)\leq Cs,\label{MoreauEnvProp:HolderConv:hhyp} \mbox{ for }s\geq 0 \mbox{ and }i\in \{1,2,3\},\eeq
    then
    \beq \lim_{\lambda\to 0^{+}}g_{\lambda}(x)=0 \mbox{ in }C_{loc}(\Omega).\label{MoreauEnvProp:HolderConv:eq3}\eeq
\end{prop}
\begin{proof}
    Let $\lambda>0$ and  $x\in \Omega^{\lambda}_{\phi}$. Then by \eqref{MoreauEnvProp:item00} we have {$\mathcal{A}_{\lambda}\phi(x)\neq\emptyset$}. Since $\phi$ is a super solution of $F(y,\nabla\phi (y))=0$ and $y\mapsto \phi(x)+\frac{1}{2\lambda}|x-y|^{2} $ attains its minimum at $y\in \mathcal{A}_{\lambda}\phi(x)$, we find
    $$F\left(y,\frac{1}{\lambda}(x-y)\right)\geq 0\mbox{ for all }y\in\mathcal{A}_{\lambda}\phi(x).$$
    By using \eqref{MoreauEnvProp:HolderConv:eq1} we have for all $y\in\mathcal{A}_{\lambda}\phi(x)$ the following
    $$ \left|F\left(x,\frac{1}{\lambda}(x-y)\right)-F\left(y,\frac{1}{\lambda}(x-y)\right)\right|\leq   h_1(|x-y|)\frac{|x-y|^{q}}{\lambda^q}+h_2(|x-y|)\frac{|x-y|}{\lambda}+h_3(|x-y|).$$
    These last two inequalities imply that
    \beq  F\left(x,\frac{1}{\lambda}(x-y)\right)+g_{\lambda }(x)\geq 0  \mbox{ for all }y\in\mathcal{A}_{\lambda}\phi(x).\label{MoreauEnvProp:HolderConv:eq4}\eeq
    If $\mathcal{M}_{\lambda }\phi$ is differentiable at $x$, then by \Cref{MoreauEnvProp}(\ref{MoreauEnvProp:item1}) we have that $\mathcal{A}_{\lambda}\phi(x)=\{y\}$ is a singleton and  $\nabla \mathcal{M}_{\lambda }(x)=\frac{1}{\lambda}(x-y)$. Combining this, the almost everywhere differentiability in $\Omega^{\lambda}_{\phi}$ of $\mathcal{M}_{\lambda}\phi$,  and \eqref{MoreauEnvProp:HolderConv:eq4} we get \eqref{MoreauEnvProp:HolderConv:eq2}.

     For \eqref{MoreauEnvProp:HolderConv:eq3}, we notice that if $\phi$ is $\alpha$-Hölder continuous, then by \eqref{MoreauEnvProp:item00:1} there exists a constant $C>0$ such that
    $$ |x-y|\left(1+\frac{|x-y|^{q}}{\lambda^{q}}+\frac{|x-y|}{\lambda}\right)\leq C\left(\lambda^{\frac{1}{2-\alpha}}+\lambda^{\frac{q+1}{2-\alpha}-q}+\lambda^{\frac{\alpha}{2-\alpha}}\right) \mbox{ for all }y\in \mathcal{A}_{\lambda}\phi(x),$$
    where the constant does not depend on $x$. Consider now $\omega\subset\Omega$ compact. There exists $\lambda_{0}>0$ such that for all $\lambda\in (0,\lambda_0)$ we have $\omega\subset\Omega^{\lambda}_{\phi}$. Then, the previous inequality and the fact that $\frac{q+1}{2-\alpha}-q>0$ for $\alpha >1-\frac{1}{q}$, imply that
    $$\lim_{\lambda\to 0^{+}}\sup_{x\in\omega,y\in\mathcal{A}_{\lambda}\phi(x)}|x-y|\left(1+
    \frac{|x-y|^{q}}{\lambda^{q}}+\frac{|x-y|}{\lambda}\right)=0.$$
    This together with condition \eqref{MoreauEnvProp:HolderConv:hhyp} concludes  the proof of \eqref{MoreauEnvProp:HolderConv:eq3}.
\end{proof}

\begin{rem}\label{rem:MoreauEnv:HJB}
We apply Proposition \ref{MoreauEnvProp:HolderConv} with  $F$ given by  the Hamiltonian in $ \mathcal{H}(y,\nabla v(y))=0\mbox{ in }\Omega$, where $v$ will be the
 value function $V\in L^\infty(\Omega)\cap C(\Omega)$ associated to \eqref{ControlProblem} which is a viscosity solution to this equation.
By \eqref{Hamiltonian:prop:2} assumption \eqref{MoreauEnvProp:HolderConv:eq1} is satisfied with $q=2$, $h_1(s)=\frac{1}{\beta}\norm{B}_{Lip(\Omega;\R^{d\times m})}s$, $h_2(s)=\norm{f}_{Lip(\Omega;\R^{d})}s$ and $h_3(s)=\norm{\ell}_{Lip(\Omega)}s$ independently of $R$.
Therefore Proposition \ref{MoreauEnvProp:HolderConv} implies that $\mathcal{M}_{\lambda}V$ is a super-solution of
    \beq \mathcal{H}(y,\nabla v(y))+h_{\lambda }(y)= 0 \mbox{ in }\Omega_V^\lambda, \label{rem:MoreauEnv:HJB:approxeq}\eeq
    with
\beq h_{\lambda }(x) =\sup_{y\in\mathcal{A}_{\lambda}V(x)} \left(\norm{\ell}_{Lip(\Omega)}+\frac{1}{\beta}\norm{B}_{Lip(\Omega;\R^{m\times d})}^2\frac{|x-y|^2}{\lambda^2}+\norm{f}_{Lip(\Omega;\R^{d})}\frac{|x-y|}{\lambda}\right)|x-y|.
\label{rem:MoreauEnv:HJB:g}\eeq
Moreover, assuming that $V$ is $\alpha$-Hölder continuous  with $\alpha\in \left(\frac{1}{2},1\right]$ in $\Omega$,  by \eqref{MoreauEnvProp:item00:1} we can bound $h_{\lambda }$ by
$$h_{\lambda }(x)\leq C\left(\norm{\ell}_{Lip(\Omega)}\lambda^{\frac{1 }{2-\alpha }} +\frac{1}{\beta}\norm{B}_{Lip(\Omega;\R^{m\times d})}^2\lambda^{\frac{2\alpha-1 }{2-\alpha }}+\norm{f}_{Lip(\Omega;\R^{d})}\lambda^{\frac{\alpha }{2-\alpha }}\right)\mbox{ for all }x\in \Omega^{\lambda}_{V},$$
where $C$ only depends on $\alpha$ and the Hölder constant of $V$ on $\Omega$. Additionally, if $\lambda\in (0,1]$, the last inequality implies
$$h_{\lambda }(x)\leq C\left(\norm{\ell}_{Lip(\Omega)} +\frac{1}{\beta}\norm{B}_{Lip(\Omega;\R^{m\times d})}^2+\norm{f}_{Lip(\Omega;\R^{d})}\right)\lambda^{\frac{2\alpha-1 }{2-\alpha }}.$$
In particular $h_\lambda$ tends to $0$ on $\Omega$ with the specified rate.
\end{rem}
\section{Statement of the problem and structural hypotheses}
\label{Sec:DefHyp}
The main purpose of this work is the construction of a sequence of approximating optimal feedback laws and to analyze their convergence. Here we provide the exact problem formulation and state the main structural hypotheses that will be needed.
In the following $\Subset$ denotes strict inclusion.
\begin{hypo}
\label{hypo:scapetime}
    Consider a locally Lipschitz continuous function $u:\Omega\mapsto \R^{m}$,
    $\omega\Subset\Omega$ open and $T>0$. For all $y_0\in\omega$, the triplet {$(u,\omega,T)$} satisfies that  $y(t;y_0,u)\in \Omega$ for all $t\in [0,T]$, where $y(\cdot;y_0,u)$ is the solution of
    \beq
    y'(t)=f(y(t))+B(y(t))u(y(t)),\quad y(0)=y_0.
    \label{closedloopproblem}
    \eeq
\end{hypo}

\begin{defi}
     For $(u,\omega,T)$ satisfying \Cref{hypo:scapetime}, we define $\mathcal{V}_{u,T}$ by
    $$y_0\in\omega\mapsto \mathcal{V}_{u,T}(y_0)=\int_{0}^{T}\left(\ell(y(t;y_0,u))+\frac{\beta}{2}|u(y(t;y_0,u))|^{2}\right)dt.$$
\end{defi}
With this definition, our objective is to prove the existence of a sequence of controls  $u_{n}\in C^{1}(\Omega;\R^{m})$ in feedback form,  and times $T_{n}$ tending to infinity such that \beq \lim_{n\to\infty} \norm{\mathcal{V}_{u_n,T_n}-V}_{L^{p}(\omega)}=0
    \label{rem:valueConv}\eeq for some $p\in [1,\infty]$ depending on the regularity of the value function $V$. One of the motivations to include the time horizon $T$ into the approximation framework, rather  than setting the time horizon to $(0,\infty)$ is motivated, besides intrinsic interest, by the fact that in the context of numerical approximation techniques for the value function $V$, many strategies rely on data for $V$ at sample points $x_i \in \Omega$. These values for $V(x_i)$ would require to solve infinite horizon optimal control problems, which is unfeasible, whereas approximations by finite horizon problems can be obtained.

    The following definition characterizes the choice of the feedback functions which will be made.

\begin{defi}
     For a function $v$ in $C^{1}(\Omega)$, we define $u_v\in C(\Omega;\R^{m})$ by
    \beq u_v(y_0)=-\frac{1}{\beta}B(y_0)^{\top}\nabla v(y_0)\mbox{, for }y_0\in\Omega. \label{uv:def}\eeq
\end{defi}
\begin{rem}
    \label{rem1}
    For $v\in C^{1}(\Omega)$, we notice that \eqref{uv:def} is equivalent to
    \beq u_v(y_0)\in \argmax_{u\in \R^{d}} H(y_0,\nabla v(y_0),u)\mbox{, for }y_0\in\Omega.\label{uv:def:rem}\eeq
\end{rem}

For such a control equation \eqref{OpenLoop} admits a local solution if  $v\in C^{1}(\Omega)$ which is unique if moreover $v \in C^{1,1}(\Omega)$.

To achieve \eqref{rem:valueConv}, we will need a stability hypothesis on the optimal trajectories of \eqref{ControlProblem}-\eqref{OpenLoop}. We express this stability condition in the viscosity sense and make use of it in sections \Cref{Sec:StabilityEst} and  \Cref{Sec:Approx}.

\begin{hypo}\label{Lyapunov}
Let $\omega\Subset\Omega$ be open and $\phi\in C(\Omega)$. There exist $\delta>0$, $w\in C^{1}(\Omega)$, and $g\in C(\Omega)$, bounded from below by 0,  such that $\omega$ is strictly contained in
 $$\omega_{\delta}:=\{y\in \Omega: w(y) < \sup_{y_0\in \omega} w(y_0)+\delta\}.$$
Further, $\omega_\delta$ is strictly contained in $\Omega$, $\partial \omega_\delta$ is of class $C^1$,
and $\phi$ is a viscosity super solution of
$$ -\nabla w(y)^{\top}(f(y)+B(y)u_\phi(y))+ g(y)=0 \mbox{ in }\omega_{\delta},$$
i.e. for every $\bar{y}\in\omega_\delta$ and every $h\in C_{loc}^{1}(\Omega)$ such that $\phi-h$ attains a local minimum at $\bar{y} $ the inequality
$$ \nabla w(\bar{y})^{\top}(f(\bar{y})+B(\bar{y})u_h(\bar{y}))\leq  g(\bar{y})$$
holds.
\end{hypo}

The existence of a Lyapunov-like function $w$ as demanded by Hypothesis \ref{Lyapunov} is discussed in \cite{CLARKE199869}. The condition will be used with $\phi=V$ or an approximation thereof.  The set $\omega_\delta$ is the region where we approximate the value function.
 The regularity  assumption for the boundary  of $\omega_\delta$ holds if $\nabla w \neq 0$ on  $\partial \omega_\delta$.

In the analysis of approximating the escape time from $\omega_\delta$, the following quantities will be required.
 For $\varepsilon>0$ we set
\beq  \sigma^{1}_{\varepsilon}=\sup_{x\in \omega_{\delta},y\in B(x,\varepsilon)} \left|g(y)-\nabla w(y)^{\top} f(y)+\nabla w(x)^{\top} f(x)\right|\label{LyapunovStabLemma:sigma1}\eeq
and
\beq  \sigma^{2}_{\varepsilon}=\sup_{x\in \omega_{\delta},y\in B(x,\varepsilon)} \left|B(x)^{\top}\nabla w(x)-B(y)^{\top}\nabla w(y)\right|.\label{LyapunovStabLemma:sigma2}\eeq
For $g=0$ these quantities tend to zero with $\varepsilon$. We also need
\beq
\sigma^{1}_{\varepsilon,\lambda}=\sup_{x\in \omega_{\delta},y\in B(x,\varepsilon)} \left|g_\lambda(y)-\nabla w(y)^{\top} f(y)+\nabla w(x)^{\top} f(x)\right|\label{LyapunovStabLemma:sigma1,l}
\eeq
where
\beq \begin{array}{l}
     \dis  g_{\lambda}(x)=g(x)+
\sup_{y\in \mathcal{A}_{\lambda}\phi (x)}\Big\{ h_{w}(|x-y|)\left(\norm{f}_{L^{\infty}(\Omega;\R^{d})}+\frac{1}{\beta}\frac{|x-y|}{\lambda}\norm{ B}^2_{L^{\infty}(\Omega;\R^{d\times m})}\right)\\
\ecart\dis+\norm{f}_{Lip(\Omega;\R^{d})}\norm{\nabla w}_{L^{\infty}(\Omega;\R^{d})}|x-y|+\frac{2}{\beta}\frac{|x-y|^{2}}{\lambda}\norm{B}_{Lip(\Omega);\R^{d\times M}}^2\norm{\nabla w}_{L^{\infty}(\Omega;\R^{d})}\Big\},
\end{array}\label{LyapunovStabLemma:glambda}\eeq
and $h_{w}$ is the modulus of continuity of $\nabla w$ in $\omega_{\delta}$, i.e.
$$ |\nabla w(x_{1})-\nabla w(x_{2})|\leq h_{w}(|x_{1}-x_{2}|) \mbox{ for all }x_1,x_2\in\omega_{\delta}.$$

In Remark \ref{rem:lyapunov:conv} we shall provide a sufficient condition guaranteeing that $\lim _{\lambda\to 0^+}g_{\lambda}=g$.

\section{Main results: approximation of viscosity solutions}
\label{MainResults}

Here we state the approximation results which are part of the main  contributions of this paper, under various assumptions on the regularity of the value function $V$.  The proofs will be given in the following sections. We start with $V\in C^1(\Omega)$.

\begin{theo}
\label{C1ConvergenceTheo}
 Let $w\in C^{1}(\overline{\Omega})$, $\omega\Subset\Omega$ open, and $\delta>0$ satisfying \Cref{Lyapunov} with $\phi=V \in C^{1}(\Omega)$ the value function of \eqref{ControlProblem} and $g=0$. Further let $u_{\varepsilon}\in Lip(\Omega)$ be a family of funtions  such that
 \beq\lim_{\varepsilon\to 0^{+}}\norm{B(u_{\varepsilon}-u_V)}_{C(\omega_{\delta};\R^{M})}=0.\label{C1ConvergenceTheo:eq1}\eeq
Choose $\kappa:(0,\infty)\to(0,\infty)$  and $s_0 >0$ such that
 \beq \lim_{s\to 0^{+}}\kappa(s)=\infty \mbox{ and } \kappa(s)s\leq \frac{\delta}{\norm{\nabla w}_{L^{\infty}(\omega_\delta;\R^d)}}\mbox{ for all }s\in (0,s_0),\label{C1ConvergenceTheo:eq2}\eeq
 and  set $\tau_{\varepsilon}=\kappa(\norm{B(u_{\varepsilon}-u_V)}_{C(\omega_{\delta};\R^{M})})$. Then, there exists $\varepsilon_{0}$ such that for all $\varepsilon\in (0,\varepsilon_0)$ the triplet $(u_{\varepsilon},\omega_{\delta},\tau_{\varepsilon})$ satisfies \Cref{hypo:scapetime},
 \beq \norm{\mathcal{V}_{u_{\varepsilon},\tau_{\varepsilon}}+V\circ y(\tau_{\varepsilon};\cdot,u_\varepsilon) -V}_{L^{\infty}(\omega)}\leq \beta\kappa(\norm{B(u_{\varepsilon}-u_V)}_{C(\omega_{\delta};\R^{M})})\norm{u_{\varepsilon}-u_V}_{C(\omega_{\delta};\R^{M})}^{2},\label{C1ConvergenceTheo:eq3}\eeq
 and consequently
 \beq \lim_{\lambda\to 0^{+}}\tau_{\varepsilon}=\infty,\mbox{ and }    \lim_{\varepsilon\to 0^{+}}\norm{\mathcal{V}_{u_{\varepsilon},\tau_{\varepsilon}}+V\circ y(\tau_{\varepsilon};\cdot,u_\varepsilon) -V}_{L^{\infty}(\omega)}=0.
 \label{C1ConvergenceTheo:eq4}
 \eeq
\end{theo}

The following result holds under a semi-convexity assumption on $V$.

\begin{theo}
\label{SemiConvexConvergenceTheo}
 Let $\Omega\subset\R^{d}$ be convex and bounded, $\omega\Subset\Omega$ open, $\delta>0$, and $w\in C^{1}(\overline{\Omega})$, such that they satisfy \Cref{Lyapunov} with $\phi=V$ the value function of \eqref{ControlProblem} and $g=0$. Assume that $V$ is semi-convex with constant $C>0$. For $\lambda>0$ and $\varepsilon>0$, set $ V_{\varepsilon,\lambda}=\rho_{\varepsilon}\ast \mathcal{M}_{\lambda}V$ with $\rho$ a smooth  mollifier. Then there exists $\lambda_0>$ such that for all $\lambda\in (0,\lambda_0)$ there exists $\varepsilon(\lambda)$ such that
 \beq
 \norm{\nabla V_{\varepsilon(\lambda),\lambda}-\nabla \mathcal{M}_{\lambda}V}^2_{L^{\infty}(\omega_{\delta;\R^{d}})}\leq \lambda,\mbox{ and }\omega_{\delta}+B(0,\varepsilon(\lambda))\subset\Omega^{\lambda}_{V}.
 \label{SemiConvexConvergenceTheo:eq1}
 \eeq
 Set $u_{\lambda}=u_{V_{\varepsilon(\lambda),\lambda}}$, choose $\kappa:(0,\infty)\to(0,\infty)$ such that
 \beq \lim_{s\to 0^{+}}\kappa(s)=\infty\mbox{ and }\lim_{s\to 0^{+}}\kappa(s)s=0, \label{SemiConvexConvergenceTheo:eq2}\eeq
 and  set \beq \tau_{\lambda}=\min\left\{\delta\left(\sigma^{1}_{\varepsilon(\lambda),\lambda}+\sigma_{\varepsilon(\lambda)}^{2}\frac{\norm{B^{\top}\nabla \mathcal{M}_{\lambda}V}_{L^{\infty}(\omega_{\delta};\R^{m})}}{\beta}\right)^{-1}, \, \kappa(\lambda)\right\}\label{SemiConvexConvergenceTheo:eq3},\eeq
 where $\sigma^{1}_{\varepsilon(\lambda),\lambda}$ and $\sigma_{\varepsilon(\lambda)}^{2}$ were defined in \eqref{LyapunovStabLemma:sigma1,l} and \eqref{LyapunovStabLemma:sigma2}, respectively.
 Then, for all $\lambda\in (0,\lambda_0)$ the triplet $(u_\lambda,\omega,\tau_\lambda$ satisfies Hypothesis \ref{hypo:scapetime}, and
 \beq \lim_{\lambda\to 0^{+}}\tau_{\lambda}=\infty,\mbox{ and }\lim_{\lambda\to 0^{+}}\norm{\mathcal{V}_{u_{\lambda},\tau_{\lambda}}+V\circ y(\tau_\lambda;\cdot,u_\lambda) -V}_{L^{\infty}(\omega)}=0.\label{SemiConvexConvergenceTheo:eq4}\eeq

\end{theo}

In the following theorem we replace the $C^{1}$ regularity of $V$  employed in Theorem \ref{C1ConvergenceTheo}  by the semi-concavity of $V$. This implies that $V$ is locally Lipschitz continuous, which in turns tell us by the Rademacher Theorem that it is differentiable almost everywhere. Hence $V$ satisfies \eqref{HJBeq} almost everywhere.

\begin{theo}
\label{SemiConcaveConvergenceTheo}
 Let $\Omega\subset\R^{d}$ be  convex and bounded, $\omega\Subset\Omega$ open, $\delta>0$, and $w\in C^{1}(\overline{\Omega})$, such that they satisfy \Cref{Lyapunov} with $g=0$ and $\phi=V\in C(\Omega)$, which is assumed to be semi-concave in $\Omega$ with constant $C>0$. Set $u_{\varepsilon}=u_{V_{\varepsilon}} $ with $V_{\varepsilon}=\rho_{\varepsilon}\ast V$ a smooth mollification of $V$, and $p\in [1,\infty)$. Let $\kappa:(0,1)\to(0,\infty)$ be such that
 \beq \lim_{s\to 0^{+}}\kappa(s)=\infty\mbox{ and }\lim_{s\to 0^{+}}\kappa(s)^{\frac{p-1}{p}}(e^{K\kappa(s)}-1)^{\frac{1}{p}}s^2=0,\label{SemiConcaveConvergenceTheo:eq3}\eeq
 where
$$ K=\left(\frac{1}{\beta}\left(mdC\norm{B}_{L^{\infty}(\omega_{\delta};\R^{m\times d})}^{2}+md^2\norm{B}^2_{Lip(\omega_{\delta};\R^{d\times m})}\norm{\nabla V}_{Lip(\omega_{\delta};\R^{d})}\right)+d\norm{f}_{Lip(\overline{\Omega};\R^{d})}\right),$$
 and  set \beq \tau_{\varepsilon}=\min\left\{\delta\left(\sigma^1_{\varepsilon}+\frac{\sigma^2_{\varepsilon}}{\beta}\norm{\nabla V}_{L^{\infty}(\omega_{\delta};R^{d})}\right)^{-1},\kappa\left(\norm{u_{\varepsilon}-u_V}_{L^{p}(\omega_{\delta};\R^{m})}\right)\right\},\label{SemiConcaveConvergenceTheo:eq4}\eeq
 where $\sigma^{1}_{\varepsilon}$ and $\sigma_{\varepsilon}^{2}$ were defined in \eqref{LyapunovStabLemma:sigma1} and \eqref{LyapunovStabLemma:sigma2}, respectively. Then there exists $\varepsilon_{0}$ such that for all $\varepsilon\in (0,\varepsilon_0)$ the triplet $(u_{\varepsilon}, \omega, \tau_{\varepsilon})$ satisfies Hypothesis \ref{hypo:scapetime}  and consequently
 \beq
\lim_{\varepsilon\to 0^{+}}\tau_{\varepsilon}=\infty  \mbox{ and }
 \lim_{\varepsilon\to 0^{+}}\norm{\mathcal{V}_{u_{\varepsilon},\tau_{\varepsilon}}+V\circ y(\tau_\varepsilon;\cdot,u_\varepsilon) -V}_{L^{p}(\omega)}=0.\label{SemiConcaveConvergenceTheo:eq5}\eeq
\end{theo}
\begin{rem}
    It might seem that $\kappa$ is difficult to find, but for instance, $s\mapsto -\frac{1}{\alpha}\log(s)$ for $s\in (0,1)$ and $\alpha>\frac{K}{p+1}$, satisfies \eqref{SemiConcaveConvergenceTheo:eq3}. It is also important to observe that this provides an upper bound on the convergence of $\mathcal{V}_{u_{\varepsilon},T_\varepsilon}$ to $V$ in $L^{p}(\omega_{\delta})$.
\end{rem}
We next relax the semi-concavity assumption of Theorem \ref{SemiConcaveConvergenceTheo}  by using a mollification of the Moreau envelope of the value function combined with Lemma \ref{LyapunovStabLemma}(c) concerning the escape time from $\omega_\delta$.
\begin{theo}
\label{HolderConvergenceTheo}
 Let $\Omega\subset\R^{d}$ be convex and bounded, $\omega\Subset\Omega$ open, $\delta>0$, and $w\in C^{1}(\overline{\Omega})$, such that they satisfy \Cref{Lyapunov} with $g=0$ and $\phi=V$, which is assumed to be $\alpha-$Hölder continuous in $\Omega$ with constant $C>0$ and $\alpha\in (\frac{1}{2},1]$. Further, assume $\nabla w$ is $\sigma$-Hölder continuous with $\sigma\in (1-\alpha,1]$. For $\lambda >0$ and $\varepsilon>0$, set $V_{\varepsilon,\lambda}=\rho_{\varepsilon}\ast \mathcal{M}_{\lambda} V$ a mollification of $\mathcal{M}_\lambda V$. Let $p\in [1,\infty)$ be fixed and $\eta:(0,1)\to(0,\infty)$ be such that $\lim_{s\to 0^+}\eta(s)=0$.
 Then there exists $\lambda_0$ such that for all $\lambda \in (0,\lambda_0)$ there exist $\varepsilon(\lambda)>0$ such that
\beq \norm{\nabla V_{\varepsilon(\lambda),\lambda}-\nabla \mathcal{M}_{\lambda} V}_{L^{2p}(\omega_{\delta})}^2\leq \eta (\lambda),\ \varepsilon(\lambda)\leq \lambda^{\frac{1}{2-\alpha}},\mbox{ and }\omega_{\delta}+B(0,\varepsilon(\lambda))\subset \Omega^{\lambda}_{V}.\label{HolderConvergenceTheo:eq0}\eeq
 Set $u_{\lambda }=u_{V_{\varepsilon(\lambda),\lambda}}$ and
let $\kappa:(0,1)\to(0,\infty)$ be such that
 \beq \lim_{s\to 0^{+}}\kappa(s)=\infty\mbox{ and }\lim_{s\to 0^{+}}\kappa(s)^{\frac{p-1}{p}}(e^{K(s)\kappa(s)}-1)^{\frac{1}{p}}\eta(s)+\kappa(s)s^{\frac{2\alpha-1}{2-\alpha}}=0,\label{HolderConvergenceTheo:eq2}\eeq
 where
\beq K(s)=\left(\frac{md(d+1)}{\beta s}\norm{B}_{Lip(\omega_{\delta};\R^{d\times M})}^{2}+d\norm{f}_{Lip(\omega_{\delta};\R^{d})}\right).\label{HolderConvergenceTheo:eq3}\eeq
 Further  set
 \beq \tau_{\lambda}=\min\left\{\delta\left(\sigma^1_{\varepsilon(\lambda),\lambda}+\frac{\sigma^2_{\varepsilon(\lambda)}}{\beta}\norm{\nabla \mathcal{M}_{\lambda}V}_{L^{\infty}(\omega_{\delta};R^{d})}\right)^{-1},\kappa\left(\lambda\right)\right\}.
 \label{HolderConvergenceTheo:eq4}\eeq
 where $\sigma^{1}_{\varepsilon(\lambda),\lambda}$ and $\sigma_{\varepsilon(\lambda)}^{2}$ were defined in \eqref{LyapunovStabLemma:sigma1,l} and \eqref{LyapunovStabLemma:sigma2}, respectively. Then for all $\lambda\in (0,\lambda_0)$ the triplet $(u_\lambda,\omega,\tau_\lambda)$ satisfies Hypothesis \ref{hypo:scapetime}
and
 \beq \lim_{\lambda\to 0^{+}}\tau_{\lambda }=\infty\mbox{ and }\lim_{\lambda\to 0^{+}}\norm{\mathcal{V}_{u_{\lambda},\tau_{\lambda}}+V\circ y(\tau_\lambda;\cdot,u_\lambda) -V}_{L^{p}(\omega)}=0.\label{HolderConvergenceTheo:eq5}\eeq

\end{theo}

\begin{rem}
    In the previous theorem, an admissible option is choosing $\kappa(s)=\frac{1}{s^{q}}$ with $q\in \left(0,\frac{2\alpha-1}{2-\alpha}\right)$ and $\eta(s) = \exp(-\frac{1}{p}K(s)\kappa(s)-\frac{1}{s^2})$. We notice that in \Cref{HolderConvergenceTheo} more smoothness on the approximation of the value function (i.e. smaller $\varepsilon(\lambda)$) implies a decrement on the bound of the error between $\mathcal{V}_{u_{\lambda},T_{\lambda}}$ and $\mathcal{ M}_{\lambda}V$.
\end{rem}
We close this section with a result concerning the convergence of the trajectories of the closed loop system for a sequence of feedback laws.

\begin{theo}
    \label{ConvTheoTraj}
    Let $\Omega\subset\R^{d}$ be a bounded subset of $\R^{d}$ and choose $y_0\in\Omega$. Consider a sequence of feedback laws $u_{\varepsilon}\in C^{1}(\overline{\Omega};\R^{m})$ and times $T_{\varepsilon}>0$ with $\lim_{\varepsilon\to 0^+}T_{\varepsilon}=\infty$ such that $y([0,T_{\varepsilon}];y_0,u_{\varepsilon})\subset\Omega$ and
    \beq \label{eq:aux5}
    \lim_{\varepsilon\to 0^{+}}\mathcal{V}_{T_{\varepsilon,u_{\varepsilon}}}(y_0)+V(y(T_\varepsilon;y_0,u_{\varepsilon}))=V(y_0),
    \eeq
     and set $y_{\varepsilon}=y(\cdot;y_0,u_{\varepsilon})$. Then, the following statements hold:
     \begin{enumerate}[(a)]
         \item \label{ConvTheoTraj:1}There exists at least one limit point of $(u_{\varepsilon}\circ y_{\varepsilon},y_{\varepsilon})$ converging in the weak topology of $L_{loc}^{2}((0,\infty);\R^{m})\times W^{1,2}_{loc}([0,\infty);{\R^d})$.
         \item \label{ConvTheoTraj:2}For every weak limit point $(u,y)\in L_{loc}^{2}((0,\infty);\R^{m})\times W^{1,2}_{loc}([0,\infty);\R^d)$ of $(u_{\varepsilon}\circ y_{\varepsilon},y_{\varepsilon})$ we have that $u$ is an optimal control of \eqref{ControlProblem} and $y$ is the associated solution of \eqref{OpenLoop}. Consequently,
         \beq \lim_{\varepsilon\to 0^{+}}\mathcal{V}_{T_{\varepsilon,u_{\varepsilon}}}(y_0)=V(y_0)\mbox{ and } \lim_{\varepsilon\to 0^{+}}V(y(T_\varepsilon;y_0,u_{\varepsilon}))=0.
         \label{ConvTheoTraj:cons:eq:1}\eeq

         \item \label{ConvTheoTraj:3}If the solution of \eqref{ControlProblem} is unique, then $u_{\varepsilon}\circ y_{\varepsilon}$ converges to the optimal solution of \eqref{ControlProblem} and $y_{\varepsilon}$ converges weakly to the solution of \eqref{OpenLoop} in $L_{loc}^{2}((0,\infty);\R^{m})$ and $W_{loc}^{1,2}((0,\infty);\R^{d})$, respectively.
     \end{enumerate}
\end{theo}

\begin{cor} \label{Corconv}
  Consider $\omega\subset\Omega$, with $\Omega$  be bounded,  $p\in [1,\infty)$, and  a sequence of feedbacks $u_{\varepsilon}\in C^{2}(\overline{\Omega};\R^{m})$ and times $T_{\varepsilon}>0$ with $\lim_{\varepsilon\to 0^+}T_{\varepsilon}=\infty$. If  we have $ y([0,T_{\varepsilon}];\omega,u_{\varepsilon})\subset\Omega$, for all $\varepsilon$ sufficiently small,  and
     \beq \lim_{\varepsilon\to 0^{+}}\norm{\mathcal{V}_{T_{\varepsilon,u_{\varepsilon}}}+V\circ y(T_\varepsilon;\cdot,u_\varepsilon) -V}_{L^{p}(\omega)}=0,\label{ConvTheoTraj:eq:1}\eeq
     then there exist a sub-sequences of $u_{\varepsilon}$ and $T_{\varepsilon}$ such that \eqref{ConvTheoTraj:1}, \eqref{ConvTheoTraj:2} and \eqref{ConvTheoTraj:3} hold for almost every $y_{0}\in \omega$. Additionally, if $p=\infty$, then \eqref{ConvTheoTraj:1}, \eqref{ConvTheoTraj:2} and \eqref{ConvTheoTraj:3} hold for every $y_0\in \omega$.
\end{cor}

\section{Error estimates}
\label{Sec:ErrorEs}
In this section we provide  estimates on the approximation properties of $\mathcal{V}_{u,T}$ towards a  super-solution of the HJB equation \eqref{HJBeq} or an approximation thereof. That is, we consider super-solutions of the following equation
\beq
    \mathcal{H}(y,\nabla v(y))+g(y)=0, \label{HJBeqApprox}
    \eeq
where $g\in C(\Omega)$. We notice that the Moreau envelope of the value function is a super-solution of an equation of this type, see \eqref{rem:MoreauEnv:HJB:approxeq}, \eqref{rem:MoreauEnv:HJB:g}. Below the results are derived in a general setting, with $v$ playing the role which will eventually be played by $V$ and an approximation  $\bar v$  of $V$.
The results differ by the regularity assumption on  $v$.  We start with an estimate for  super-solutions of class $C^{1}$. For the convenience of the reader we recall that  $u_{\bar v}(y)=-\frac{1}{\beta}B^{\top}(\bar{y})\nabla \bar v(\bar{y})$.

\begin{lemma}
\label{lemma:main:estimate}
     Let $v\in C^{1}(\Omega)$ be a super-solution of \eqref{HJBeqApprox} with $g\in C(\overline{\Omega})$. Consider $\bar{v}\in C^{1,1}(\Omega)$, $y_0\in \Omega$, and $T$ such that $y(\cdot;y_0,u_{\bar{v}})$ exists on $[0,T]$ and satisfies   $y([0,T];y_0,u_{\bar{v}})\subset\Omega$. Then we have
    \beq
    \mathcal{V}_{u_{\bar{v}},T}(y_0)+v(y(T;y_0,u_{\bar{v}}))-v(y_0)\leq \int_{0}^{T}g(y(t;y_0,u_{\bar{v}})) dt+\beta\int_{0}^{T}|u_{v}(y(t;y_0,u_{\bar{v}}))-u_{\bar{v}}(y(t;y_0,u_{\bar{v}}))|^{2}dt
    \label{lemma:main:estimate:eq}
    \eeq
\end{lemma}

\begin{proof}
    To abbreviate the notation, we set $\bar{y}(t)=y(t;y_0,u_{\bar{v}})$, $\bar{u}(t)=u_{\bar{v}}(\bar{y}(t))$, and $u(t)=u_{v}(\bar{y}(t))$. In particular, $\bar{y}$ satisfies $\bar{y}'=f(\bar{y})+B(\bar{y})\bar{u}$ and $\bar{y}(0)=y_0$. Here and below the dependence of the state and control variables on $t$ is not indicated. By  \eqref{uv:def:rem} we find
    $$ \nabla \bar{v}(\bar{y})(f(\bar{y})+B(\bar{y})\bar{u}) + \ell (\bar{y})+\frac{\beta}{2}|\bar{u}|^{2}\leq \nabla \bar v(\bar{y})(f(\bar{y})+B(\bar{y})u) + \ell (\bar{y})+\frac{\beta}{2}|u|^{2} \text{ for all } u\in \R^m.$$
    Using that $v$ is a super-solution of \eqref{HJBeqApprox} on the right-hand side of the previous inequality we obtain
    \beq \nabla \bar{v}(\bar{y})(f(\bar{y})+B(\bar{y})\bar{u}) + \ell (\bar{y})+\frac{\beta}{2}|\bar{u}|^{2}\leq (\nabla \bar{v}(\bar{y})-\nabla v(\bar{y}))(f(\bar{y})+B(\bar{y})u)+g(\bar{y}).\label{lemma:main:estimate:proof:1}\eeq
    Here we used that $u=-\frac{1}{\beta}B^{\top}(\bar{y})\nabla v(\bar{y})$. This and $\bar{u}=-\frac{1}{\beta}B^{\top}(\bar{y})\nabla \bar{v}(\bar{y})$, will  be used in the following equality
    $$\begin{array}{l}
         \dis (\nabla \bar{v}(\bar{y})-\nabla v(\bar{y})))(f(\bar{y})+B(\bar{y})u)= \frac{1}{\beta}|B(\bar{y})^{\top}(\nabla v(\bar{y})-\nabla \bar{v}(\bar{y}))|^{2}+ (\nabla \bar{v}(\bar{y})-\nabla v(\bar{y}))(f(\bar{y})+B(\bar{y})\bar{u})\\
        \ecart\dis =\beta|u-\bar{u}|^{2}+ (\nabla \bar{v}(\bar{y})-\nabla v(\bar{y}))(f(\bar{y})+B(\bar{y})\bar{u}).
    \end{array}$$
    Together with \eqref{lemma:main:estimate:proof:1} we arrive at
    $$ \nabla v(\bar{y})(f(\bar{y})+B(\bar{y})\bar{u}) + \ell (\bar{y})+\frac{\beta}{2}|\bar{u}|^{2}\leq \beta|u-\bar{u}|^{2}+g(\bar{y}).$$
    Integrating from $0$ to $T$ in the previous inequality we arrive at \eqref{lemma:main:estimate:eq}.
\end{proof}
The next theorem is a direct consequence of this lemma. It  gives an $L^{\infty}$ estimation of the approximation property of $\mathcal{V}_{T,\bar{v}}$, for $\bar v$ of class $C^{1,1}$,  and a super solution of \eqref{HJBeqApprox} of class $C^1$.

\begin{theo}
\label{ConvTheo1}
Let $ {\Omega_1} \subset \R^d$  be an open and bounded set with $C^1$ boundary, $\omega\Subset\Omega_1$ open, $v\in C^{1}(\Omega_1)$ be a super-solution of \eqref{HJBeqApprox} in $\Omega_1$ with $g\in C(\overline{\Omega_1})$ and let \Cref{hypo:scapetime} hold with $u=u_{\bar{v}}$ for $\bar{v}\in C^{1,1}(\overline{\Omega_1})$. Then we have
$$\sup_{y_0\in\omega}\left\{\mathcal{V}_{u_{\bar{v}},T}(y_0)+v( y(T;\cdot,u_{\bar{v}})(y_0)) -v(y_0)\right\}\leq T\left(\sup_{x\in \Omega_1} g(x)^{+}+ \beta \norm{u_{\bar{v}}-u_v}^{2}_{L^{\infty}(\Omega_1)}\right).$$
\end{theo}

Our next estimation is of vital importance for the case when the value function is not differentiable. It will be used with $ \Omega_1 = \omega_\delta.$

\begin{theo}
\label{ConvTheo2}
Let $ {\Omega_1} \subset \Omega$  be an open and bounded set with $C^1$ boundary, $\omega\Subset\Omega_1$ open, $v\in Lip(\overline{\Omega_1})$ be a super-solution of \eqref{HJBeqApprox} in $\Omega_1$ with $g\in C(\overline{ \Omega_1})$, and  let $\bar{v}\in C^{2}(\overline{\Omega_1})$ be such that for some constant $C>0$
\beq -tr(B(y)Du_{\bar{v}}(y))\leq C\mbox{ for all }y\in\Omega_1.
\label{ConvTheo2:eq0}\eeq
Let $\omega\Subset\Omega_1$ and \Cref{hypo:scapetime} holds true  with $u=u_{\bar{v}}$, then  for all $p\in [1,\infty)$ the following inequality holds
$$\norm{\left(\mathcal{V}_{u_{\bar{v}},T}+v( y(T;\cdot,u_{\bar v})) -v\right)^{+}}_{L^{p}(\omega)}\leq T|\omega|^{\frac{1}{p}}\norm{g}_{L^{\infty}(\Omega_1)}+T^{\frac{p-1}{p}}\beta\left(\frac{e^{KT}-1}{K}\right)^{\frac{1}{p}} \norm{u_v-u_{\bar{v}}}^{2}_{L^{2p}(\Omega_1;\R^{m})},$$
where
\beq K=\left(C+d\norm{f}_{Lip(\overline{\Omega_1};\R^{d})}+dm\norm{B}_{Lip(\overline{\Omega_1};\R^{d\times m})}\norm{u}_{L^{\infty}(\Omega_1;\R^{m})}\right).\label{ConvTheo2:K}\eeq
\end{theo}
This result is based on the following technical lemma which concerns the integration along the trajectories of the closed loop problem \eqref{closedloopproblem} for a $C^{1}$ feedback law. The proof of \Cref{ConvTheo2} is given after the proof of this lemma.
\begin{lemma}
\label{LemmaInt}
    Let $\Omega_1\Subset\Omega$ open and  $\omega\Subset\Omega_1$, let $u\in C^{2}(\overline{\Omega_1};\R^{M})$ be such that there exist a positive constant $C>0$ satisfying
    \beq -tr(B(y)Du(y))\leq C \mbox{ in }\Omega_1\label{LemmaInt:eq1},\eeq
    and let \Cref{hypo:scapetime} with $u=u_{\bar{v}}$ hold true. Then for all $\phi\in C(\overline{\Omega_1};\R^{+})$
    $$ \int_{\omega}\int_{0}^{T}\phi(y(t;y_0,u_{\bar{v}})) dt dy_0 \leq  \frac{e^{KT}-1}{K}\int_{\Omega_1}\phi(z)dz$$
   holds,  where
    $$ K=\left(C+d\norm{f}_{Lip(\overline{\Omega_1};\R^{d})}+dm\norm{B}_{Lip(\overline{\Omega_1};\R^{d\times m})}\norm{u_{\bar{v}}}_{L^{\infty}(\Omega_1;\R^{m})}\right),$$
    and $\R^{m}$ is endowed with the $\ell^\infty$-norm.
\end{lemma}
\begin{proof}[Proof of \Cref{LemmaInt}]
We know that for all $t\in [0,T]$ the mapping $y_0\mapsto y(t;y_0,u)$ is $C^{1}$ in $\omega$, and its differential satisfies
$$ \frac{d}{dt}D_{y_0}y(t;y_0,u)= A(y(t;y_0,u))\cdot D_{y_0}y(t;y_0,u) , \  D_{y_0}y(0;y_0,u)=I_{d\times d}, $$
where $u=u_{\bar v}$, and the components of $A(t,y_0)$ are given by

$$A_{i,r}(y)=\frac{\partial f_i }{\partial y_r}(y)+\sum_{j=1}^{m}\left(\frac{\partial B_{i,j}}{\partial y_r}(y)u_j(y)+ B_{i,j}(y)\frac{\partial u_j}{\partial y_r}(y)\right).$$
By classical ODE theory we know that
$$ det(D_{y_0}y(t;y_0,u))=\exp\left(\int_{0}^{t} tr(A(y(s;y_0,u))) ds \right).$$
Using \eqref{LemmaInt:eq1} in the previous expression we get
        \beq  det(D_{y_0}y(t;y_0,u))\geq exp\left( -t K\right)\mbox{ for all }t\in [0,T].\label{LemmaInt:eq2}\eeq
By the Fubini theorem and using the change of variable $z=y(t;y_0,u)$ we have
$$ \int_{\omega}\int_{0}^{T}\phi(y(t;y_0,u)) dt dy_0= \int_{0}^{T}\int_{y(t;\omega,u)}\frac{\phi(z)}{det(D_{y_0}y(t; y^{-1}(t;z,u),u))} dz dt, $$
where $z\mapsto y^{-1}(t;z,u)$ denotes the inverse function of $y_0\mapsto y(t;y_0,u)$. Now, using \eqref{LemmaInt:eq2} we get
$$ \int_{\omega}\int_{0}^{T}\phi(y(t;y_0,v)) dt dy_0\leq  \int_{0}^{T}\int_{y(t;\omega,u)}\phi(z)e^{Kt} dz dt=\frac{e^{KT}-1}{K} \int_{\Omega_1}\phi (z) dz,$$
which concludes the proof.
\end{proof}

\begin{proof}[Proof of \Cref{ConvTheo2}]
Since by assumption $\partial \Omega_1$ is $C^1$ regular,  there exists a family of functions $v_{\varepsilon}$ in $C^{1}(\overline{\Omega_1})$  such that
\beq
\lim_{\varepsilon\to 0}\norm{v-v_{\varepsilon}}_{W^{1,2p}(\Omega_1)}=0\mbox{ and }\ \lim_{\varepsilon\to 0}\norm{v-v_{\varepsilon}}_{C(\overline{\Omega_1})}=0,
\label{ConvTheo2:eq1}\eeq
see \cite[Theorem 3, Section 5.3.3]{Evans}.  Then, by \eqref{Hamiltonian:prop}, \eqref{ConvTheo2:eq1}, and boundedness of $\Omega_1$   we get
\beq
\begin{array}{l}
   \dis \lim_{\varepsilon\to 0^+} H(y,\nabla v_{\varepsilon}(y),u_{v_{\varepsilon}}(y))
     =\lim_{\varepsilon\to 0^+}\mathcal{H}(y,\nabla v_{\varepsilon}(y))
     =\mathcal H(y,\nabla v(y))\mbox{ in }L^{p}(\Omega_1).
\end{array}
\label{ConvTheo2:eq2}
\eeq
Moreover we have
\begin{equation}\label{eq:aux4}
H(y,\nabla v_{\varepsilon}(y),u_{v_{\varepsilon}}(y))+h_{\varepsilon}(y)+g(y)
=\mathcal{H}(y,\nabla v_{\varepsilon}(y))+h_{\varepsilon}(y)+g(y)\geq 0,
\end{equation}
where $h_{\varepsilon}\in C(\overline{\Omega_1})$ is given by $ h_{\varepsilon}(y)= -\min\{0,\mathcal{H}(y,\nabla v_{\varepsilon}(y))+g(y)\}$ .
Thanks to the assumption  that $v$ is a super-solution of \eqref{HJBeqApprox} and by \eqref{ConvTheo2:eq2}, we have
\beq\lim_{\varepsilon\to 0^{+}}h_{\varepsilon}(y)= 0\mbox{ in }L^{p}(\Omega_1).\label{ConvTheo2:eq3}\eeq
To simplify  notation, we now denote $y(t;y_0,u_{\bar v})$ by $y(t)$. Applying \Cref{lemma:main:estimate}, with $v=v_\epsilon$,  and using \eqref{eq:aux4} we find that
\beq \begin{array}{l}
\dis \mathcal{V}_{u_{\bar{v}},T}(y_0)+v_{\varepsilon}(y(T))-v_{\varepsilon}(y_0)\leq \int_{0}^{T} (h_{\varepsilon}(y(t))+g(y(t))+\beta|u_{v_{\varepsilon}}(y(t))-u_{\bar{v}}(y(t))|^{2})dt.
\end{array}
\label{ConvTheo2:eq5}
\eeq
To finish the proof we need to apply the ${L^{p}(\Omega_1)}$-norm on both sides of the inequality. By the  Minkowsky inequality and the  Jensen inequality we have

$$\begin{array}{l}
     \dis  \left(\int_{\omega}\left|\int_{0}^{T} (h_{\varepsilon}(y(t))+g(y(t))+\beta|u_{v_{\varepsilon}}(y(t))-u_{\bar{v}}(y(t))|^{2})dt\right|^{p}dy_0\right)^{\frac{1}{p}}\\
     \ecart\dis\leq \left(\int_{\omega}\left|\int_{0}^{T}g(y(t))dt\right|^{p}dy_0\right)^{\frac{1}{p}}+\left(\int_{\omega}\left|\int_{0}^{T}h_\varepsilon(y(t))dt\right|^{p}dy_0\right)^{\frac{1}{p}}+
     \left(\int_{\omega} \left|\int_{0}^{T}\beta|u_{v_{\varepsilon}}(y(t))-u_{\bar{v}}(y(t))|^{2}dt\right|^{p}dy_0\right)^{\frac{1}{p}}\\
      \ecart\dis\leq T|\omega|^{\frac{1}{p}}\norm{g}_{L^{\infty}(\Omega_1)}+T^{\frac{p-1}{p}}\left(\int_{\omega}\int_{0}^{T}|h_\varepsilon(y(t))|^{p}dt dy_0\right)^{\frac{1}{p}}+
     \beta T^{\frac{p-1}{p}}\left(\int_{\omega} \int_{0}^{T}|u_{v_{\varepsilon}}(y(t))-u_{\bar{v}}(y(t))|^{2p}dtdy_0\right)^{\frac{1}{p}}.
\end{array}   $$
Then, using \Cref{LemmaInt} we obtain

$$\begin{array}{l}
     \dis  \left(\int_{\omega}\left(\int_{0}^{T} (h_{\varepsilon}(y(t))+g(y(t))+\beta|u_{v_{\varepsilon}}(t)-u_{\bar{v}}(t)|^{2})dt\right)^{p}dy_0\right)^{\frac{1}{p}}\\
      \ecart\dis\leq T|\omega|^{\frac{1}{p}}\norm{g}_{L^{\infty}(\Omega_1)}+
     T^{\frac{p-1}{p}}\left(\frac{e^{KT}-1}{K}\right)^{\frac{1}{p}}\left(\norm{h_\varepsilon}_{L^{p}(\Omega_1)}+\beta\norm{u_{v_{\varepsilon}}-u_{\bar{v}}}^{2}_{L^{2p}(\Omega_1)}\right).
\end{array}   $$
By using this  in \eqref{ConvTheo2:eq5} we get
\beq \begin{array}{l}
\dis \norm{\left(\mathcal{V}_{u_{\bar{v}},T}+v_{\varepsilon}\circ y(T;\cdot,u_{\bar v})-v_{\varepsilon}\right)^{+}}_{L^{p}(\omega)}\\
\ecart\dis\leq T|\omega|^{\frac{1}{p}}\norm{g}_{L^{\infty}(\Omega_1)}+
     T^{\frac{p-1}{p}}\left(\frac{e^{KT}-1}{K}\right)^{\frac{1}{p}}\left(\norm{h_\varepsilon}_{L^{p}(\Omega_1)}+\beta\norm{u_{v_{\varepsilon}}-u_{\bar{v}}}^{2}_{L^{2p}(\Omega_1)}\right)
\end{array}
\label{ConvTheo2:eq8}
\eeq
Letting $\varepsilon\to 0^{+}$ and thanks to \eqref{ConvTheo2:eq8}, \eqref{ConvTheo2:eq1} and \eqref{ConvTheo2:eq3} we get the desired result.
\end{proof}
 \begin{rem}
   It is worth to mentioning that these error bounds could be applied to study the convergence of data-driven approaches. Consider $\Omega_1$, $T$, and $\omega$ as in \Cref{ConvTheo2}, and  $v_\theta\in C^2(\Omega)$ obtained by some data-driven approach as an approximation of the value function $V$ where $\theta$ are the parameters of this method. Setting $u_\theta=u_{v_\theta}$,  \Cref{ConvTheo2} leads to an $L^{p}(\omega)$ error bound between $\mathcal{V}_{u_\theta}$ and $V$ depending on the semi-concavity of $v_\theta$ as long as \Cref{hypo:scapetime} is met with $u=u_\theta$. Hence, in order to use this error bound for data-driven approaches it would be important to study the capability of these approaches of ensuring \Cref{hypo:scapetime} and  controlling the semi-concavity of $v_\theta$.
\end{rem}

\section{Escape time estimates}
\label{Sec:StabilityEst}
In \Cref{Sec:ErrorEs} we have assumed that the state is bounded until some time $T>0$, see \Cref{hypo:scapetime}. Following the notation of \Cref{Lyapunov}, here we give an estimate from below for the escape time of the trajectories of \eqref{closedloopproblem} for a feedback law given by an approximation of $\phi$. We consider three cases depending on the regularity of $\phi$. The first one corresponds to $\phi\in C^{1}(\Omega)$, the second one the semi-concavity of $\phi$ and the last one concerns the general case with $\phi$ continuous. In the context of proving the results of \Cref{MainResults}, we aim to apply the results of this section with $V$ playing the role of $\phi$.
\begin{lemma}
\label{LyapunovStabLemma}
Let $\omega\subset\Omega$, $\phi\in C(\Omega)$ and $\delta>0$ such that they satisfy \Cref{Lyapunov}.
\begin{enumerate}[(a)]
    \item \label{LyapunovStabLemma:a} If $\phi\in C^{1}(\Omega)$, consider  $u\in Lip(\Omega;\R^{m})$, and let $\hat{T}$ be the maximum $T>0$ such that $y([0,T];\omega,u)\subset \omega_\delta$. Then the following holds
    \beq  \hat{T}\left(\norm{B(u-u_{\phi})}_{C(\omega_{\delta};\R^{M})} \norm{\nabla w}_{C(\omega_{\delta};\R^{d})}+\max_{x\in\omega_{\delta}} g(x)\right)\geq \delta.\label{LyapunovStabLemma:eq0}\eeq
    \item \label{LyapunovStabLemma:b} If $\phi\in Lip(\Omega)$, set $u_{\varepsilon}= u_{\phi_\varepsilon}$ with $\phi_{\varepsilon}=\phi\ast \rho_{\varepsilon}$ a mollification of $\phi$, and let $T_{\varepsilon}$ be the maximum $T$ such that $y([0,T];\omega,u_{\varepsilon})\subset \omega_{\delta}$. Then there exists $\varepsilon_0$ such that  all $\varepsilon\in (0,\varepsilon_0)$ we have
\beq  T_{\varepsilon}\left(\sigma^{1}_{\varepsilon}+\frac{\sigma^{2}_{\varepsilon}}{\beta}\norm{B^{\top}\nabla \phi}_{L^{\infty}(\omega_{\delta};\R^{d})}\right)\geq \delta.\label{LyapunovStabLemma2:eq0}\eeq
where $\sigma_\varepsilon^1$ is defined in \eqref{LyapunovStabLemma:sigma1} and $\sigma_\varepsilon^2$ in \eqref{LyapunovStabLemma:sigma2}.
    \item \label{LyapunovStabLemma:c}Let $\phi\in C(\Omega)$ and set $u_{\varepsilon,\lambda}= u_{\phi_{\varepsilon,\lambda}}$, where $\phi_{\varepsilon,\lambda}=\mathcal{M}_{\lambda}\phi\ast \rho_{\varepsilon}$ is a mollification of $\mathcal{M}_{\lambda}\phi$, and let $T_{\varepsilon,\lambda}$ be the maximum $T$ such that $y(t;\omega,u_{\varepsilon,\lambda})\subset \omega_\delta$ for all $t\in [0,T]$. Then there exists $\lambda_0>0$ such that for all $\lambda\in (0,\lambda_0)$ there exits $\varepsilon(\lambda)$  satisfying that for all $\varepsilon\in (0,\varepsilon(\lambda))$
\beq  T_{\varepsilon,\lambda}\cdot\left(\sigma^1_{\varepsilon,\lambda}+
\frac{\sigma^2_{\varepsilon}}{\beta}\norm{B^{\top}\nabla \mathcal{M}_{\lambda}\phi}_{L^{\infty}{(\omega_{\delta})}}\right)\geq \delta,
\label{LyapunovStabLemma3:eq0}\eeq
where $\sigma^{2}_{\varepsilon}$ is given in \eqref{LyapunovStabLemma:sigma2} and  $\sigma^{1}_{\varepsilon,\lambda}$ is given by \eqref{LyapunovStabLemma:sigma1,l}.
\end{enumerate}
\end{lemma}

\begin{proof}
\begin{enumerate}
    \item \emph{Proof of (a)} Let $y_0\in \omega$ be arbitrary and denote $y(t)=y(t;y_0,u)$ for $t\in [0,\hat{T}]$. By \Cref{Lyapunov} we deduce
    $$ \begin{array}{l}
         \dis \nabla w(y(t))^{\top}(f(y(t))+B(y(t))u(y(t)))\\
         \ecart \dis \leq \nabla w(y(t))^{\top}(f(y(t))+B(y(t))u_\phi(y(t)))+\nabla w(y(t))^{\top}B(y(t))(u(y(t))-u_\phi(y(t)))\\
         \ecart\dis \leq g(y(t))+\nabla w(y(t))^{\top}B(y(t))(u(y(t))-u_\phi(y(t))).
    \end{array}$$
    Integrating from $0$ to $t$ with $t\in [0,\hat{T}]$ we obtain
    $$\begin{array}{l}
         \dis w(y(t))-w(y_0)\leq \int_{0}^{t} \left(g(y(s))+\nabla w(y(s))^{\top}B(y(s))(u(y(s))-u_\phi(y(s)))\right)ds \\
         \ecart\dis \leq t\left(\sup_{x\in \omega_{\delta}}g(x)+\norm{B(u-u_\phi)}_{L^{\infty}(\omega_{\delta};\R^{d})}\norm{\nabla w}_{L^{\infty}(\omega_{\delta};\R^{d})} \right).
    \end{array}  $$
    If $\hat{T}$ does not satisfy \eqref{LyapunovStabLemma:eq0}, then we have $\hat{T}<+\infty$ and by the previous inequality there exists $y_0\in\omega$ such that
    $$  w(y(\hat{T}))-w(y_0)<\delta.$$
    Thus, by continuity there exists $\tilde{T}>\hat{T}$ such that $w(y(\tilde{T}))-w(y_0)\leq\delta$ and  $y([0,\tilde{T}];y_0,u)\subset \omega_\delta$. Since this contradicts the definition of $\hat{T}$, we obtain that \eqref{LyapunovStabLemma:eq0} holds.
\item \emph{Proof of (\ref{LyapunovStabLemma:b})} Since  $\omega_{\delta}\Subset\Omega$, there exists $\varepsilon_0$ such that for all $\varepsilon<\varepsilon_0$ we have
$\omega_{\delta}+B(0,\varepsilon)\subset \Omega$. Consider $\rho_{\varepsilon}$ a family  of mollifiers. Since $\rho_{\varepsilon}$ are positive and \eqref{Lyapunov} we have
$$ \nabla w^{\top} (f+Bu_\phi)\ast \rho_{\varepsilon}(x)\leq g\ast \rho_{\varepsilon}(x) \mbox{ for all }x\in \Omega_{\varepsilon},$$
with $\Omega_{\varepsilon}$ defined as in \Cref{prop:hessbound}. Then for $x\in \omega_{\delta}$ we can write
\beq \begin{array}{l}
\dis \nabla w(x)^{\top}(f(x)+B(x)u_{\varepsilon}(x))\leq  \nabla w(x)^{\top}B(x)u_{\varepsilon}(x)-\left(\nabla w^{\top}Bu_\phi\right)\ast \rho_{\varepsilon}(x) \\
\ecart\dis+g\ast \rho_{\varepsilon}(x)+\nabla w^{\top}(x) f(x)-\nabla w^{\top}f\ast \rho_{\varepsilon}(x) .
\end{array}\label{LyapunovStabLemma2:eq1}\eeq
 We can bound the first two terms in the right-hand side of \eqref{LyapunovStabLemma2:eq1} as follows
\beq \begin{array}{l}
     \dis \nabla w(x)^{\top}B(x)u_{\varepsilon}(x)+\frac{1}{\beta}\int_{B(x,\varepsilon)}\nabla w(y)^{\top}B(y)B^{\top}(y)\nabla \phi(y)\rho_{\varepsilon}(x-y) dy\\
     \ecart\dis =-\frac{1}{\beta}\int_{B(x,\varepsilon)}\left(\nabla w(x)^{\top}B(x)B^{\top}(y)-\nabla w(y)^{\top}B(y)B^{\top}(y)\right)\nabla \phi(y)\rho_{\varepsilon}(x-y) dy\\
     \ecart\dis \leq \frac{1}{\beta}\sup_{y\in B(x,\varepsilon)}\left|B^{\top}(x)\nabla w(x)^{\top}-B(x)^{\top}(y)\nabla w(y)\right|\cdot\norm{B^{\top}\nabla \phi}_{L^{\infty}(\omega_{\delta};\R^{d})}=\frac{\sigma_{\varepsilon}^{2}}{\beta} \norm{B^{\top}\nabla \phi}_{L^{\infty}(\omega_{\delta};\R^{d})}.
\end{array}\label{LyapunovStabLemma2:eq2}\eeq
For the remaining terms of the right-hand side of \eqref{LyapunovStabLemma2:eq1} we have
\beq g\ast \rho_{\varepsilon}(x)+\nabla w^{\top}(x) f(x)-\nabla w^{\top}f\ast \rho_{\varepsilon}(x)\leq \sigma^{1}_\varepsilon .\label{LyapunovStabLemma2:eq3}\eeq
Then by using \eqref{LyapunovStabLemma2:eq2} and \eqref{LyapunovStabLemma2:eq3} in \eqref{LyapunovStabLemma2:eq1}  we to arrive to
$$   \nabla w(x)^{\top}(f(x)+B(x)u_{\varepsilon}(x)) \leq (\sigma^{1}_{\varepsilon}+\frac{\sigma_{\varepsilon}^{2}}{\beta} \norm{B^{\top}\nabla \phi}_{L^{\infty}(\omega_{\delta};\R^{d})}).$$
Integrating from 0 to $t\in (0,T_{\varepsilon})$ we obtain
$$ w(y(t;y_0,u_\varepsilon))\leq w(y_0)+t(\sigma^{1}_{\varepsilon}+\frac{\sigma_{\varepsilon}^{2}}{\beta} \norm{B^{\top}\nabla \phi}_{L^{\infty}(\omega_{\delta};\R^{d})})$$
for all $t\in [0,T_\varepsilon]$ and all $y_0\in \omega$. In particular, if $T_{\varepsilon}$ does not satisfy \eqref{LyapunovStabLemma2:eq0}, then $T_{\varepsilon}<\infty$ and there exists $y_0\in\omega$ such that
$$ w(y(T_{\varepsilon};y_0,u_\varepsilon))< w(y_0)+\delta.$$
From this, \eqref{LyapunovStabLemma2:eq0}  follows as in the end of the proof of (\ref{LyapunovStabLemma:a}).
\item \emph{Proof of (\ref{LyapunovStabLemma:c})} Since $\omega_\delta\subset\Omega$ there exists $\lambda_0$ such that for all $\lambda\in (0,\lambda_0)$  we have that $\omega_\delta\Subset\Omega_\phi^{\lambda}$ which in turn implies that there exists $\varepsilon(\lambda)$ satisfying that for all $\varepsilon\in (0,\varepsilon(\lambda))$ it holds that $\omega_\delta+B(0,\varepsilon)\Subset\Omega_\phi^{\lambda}$.
     Aiming for an application of Proposition \ref{MoreauEnvProp:HolderConv} we define for
     $x\in \overline{\Omega}$ and $p\in \R^{d}$ we define $$F(x,p)=-\nabla w(x)\left(f(x)-\frac{1}{\beta}B(x)B^{\top}(x)p\right).$$
We estimate  for $x_{1},x_{2}\in\Omega$ and $p\in\R^{d}$
    $$ \begin{array}{l}
       \dis |F(x_1,p)-F(x_2,p)|\leq \Big\{ h_{w}(|x_1-x_2|)\left(\norm{f}_{L^{\infty}(\Omega;\R^{d})}+ \frac{1}{\beta}|p|\norm{B}_{L^{\infty}(\Omega;\R^{d\times m})}^2\right)\\
\ecart\dis+\norm{f}_{Lip(\Omega;\R^{d})}\norm{\nabla w}_{L^{\infty}(\Omega;\R^{d})}|x_1-y_2|+\frac{2}{\beta}|x_{1}-x_2|\cdot|p|\norm{B}_{Lip(\Omega);\R^{d\times M}}^2\norm{\nabla w}_{L^{\infty}(\Omega;\R^{d})}\Big\}.
\end{array}$$
For $s\geq 0$ we define $h_1(s)=0$, $$h_2(s)= \frac{1}{\beta} h_w(s)\norm{B}_{L^{\infty}(\Omega;\R^{d\times n})}^2+\frac{2s}{\beta} \norm{B}_{Lip(\Omega;\R^{d\times m})}^2\norm{\nabla w}_{L^{\infty}(\Omega;\R^d)}$$ and $$h_3(s)=h_\omega(s)\norm{f}_{L^{\infty}(\Omega);\R^d}+s\norm{f}_{Lip(\Omega;\R^d)}\norm{\nabla w}_{L^{\infty}(\Omega;\R^{d})}.$$
We notice that $F$ and $h_1$, $h_2$ and $h_3$ satisfy \eqref{MoreauEnvProp:HolderConv:eq1}. Combining this with the fact that $\phi$ and $w$ satisfy \Cref{Lyapunov} and using \Cref{MoreauEnvProp:HolderConv} we obtain that $\mathcal{M}_\lambda\phi$ is a viscosity super-solution of
    $$ -\nabla w(y)^{\top}(f(y)-\frac{1}{\beta}B(y)B^{\top}(y)\nabla \mathcal{M}_{\lambda}\phi(y))+g_{\lambda}(y)=0\mbox{ in }\omega_{\delta}.$$
    Therefore $\mathcal{M_{\lambda}\phi}$ and $w$ satisfy \Cref{Lyapunov} with $g=g_{\lambda}$, which enables us to use (\ref{LyapunovStabLemma:b}) to conclude \eqref{LyapunovStabLemma3:eq0}.
\end{enumerate}
\end{proof}

\begin{rem}\label{rem:lyapunov:conv}
    In \Cref{LyapunovStabLemma}\eqref{LyapunovStabLemma:c}, we notice that in view of  \Cref{MoreauEnvProp}, $g_{\lambda}$ converges to $g$ in $C(\omega_{\delta})$ if $\nabla w$ is $\sigma-$Hölder continuous and $\phi$ is $\alpha$-Hölder continuous with $\sigma\in \left(1-\alpha,1\right]$. Indeed
    $$ \begin{array}{l}
     \dis  |g_{\lambda}(x)-g(x)|\leq
\sup_{y\in \mathcal{A}_{\lambda}\phi (x)}\Big\{ C\left(\norm{f}_{L^{\infty}(\Omega;\R^{d})}|x-y|^{\sigma}+\frac{1}{\beta}\frac{|x-y|^{\sigma+1}}{\lambda}\norm{ B}^2_{L^{\infty}(\Omega;\R^{d\times m})}\right)\\
\ecart\dis+\norm{f}_{Lip(\Omega;\R^{d})}\norm{\nabla w}_{L^{\infty}(\Omega;\R^{d})}|x-y|+\frac{2}{\beta}\frac{|x-y|^{2}}{\lambda}\norm{B}_{Lip(\Omega);\R^{d\times M}}^2\norm{\nabla w}_{L^{\infty}(\Omega;\R^{d})}\Big\}.
\end{array}$$
By \eqref{MoreauEnvProp:item00:1} we have $|x-y|\leq C\lambda^{\frac{1}{2-\alpha}}$, then using this in the previous inequality we have
\begin{equation}\label{rem:lyapunov:conv:eq1}
 \begin{array}{l}
     \dis  |g_{\lambda}(x)-g(x)|\leq
\Big\{ C\left(\norm{f}_{L^{\infty}(\Omega;\R^{d})}\lambda^{\frac{\sigma}{2-\alpha}}+ \frac{1}{\beta} \lambda^{\frac{\sigma+\alpha-1}{2-\alpha}}\norm{ B}^2_{L^{\infty}(\Omega;\R^{d\times m})}\right)\\
\ecart\dis+\norm{f}_{Lip(\Omega;\R^{d})}\norm{\nabla w}_{L^{\infty}(\Omega;\R^{d})}\lambda^{\frac{2}{2-\alpha}}+\frac{2}{\beta}\lambda^{\frac{\alpha}{2-\alpha}}\norm{B}_{Lip(\Omega);\R^{d\times M}}^2\norm{\nabla w}_{L^{\infty}(\Omega;\R^{d})}\Big\},
\end{array}
\end{equation}
Then, if $\sigma\in (1-\alpha,1]$ we obtain that $|g_{\lambda}(x)-g(x)|$ tends to 0 as $\lambda$ goes to $0^{+}$ uniformly in $\omega_{\delta}$. Moreover, if $\phi$ is Lipschitz continuous, then we do not need to impose further regularity on $w$. The reason for this to hold is that by \eqref{MoreauEnvProp:item00:1} we have $\frac{|x-y|}{\lambda}\leq 2\norm{\nabla \phi}_{L^{\infty(\omega_{\delta};\R^{d})}},$ which implies
$$ \begin{array}{l}
     \dis  |g_{\lambda}(x)-g(x)|\leq
\sup_{y\in \mathcal{A}_{\lambda}\phi (x)}h_{w}(|x-y|)\Big\{ C\left(\norm{f}_{L^{\infty}(\Omega;\R^{d})}+\frac{1}{\beta}\norm{\nabla \phi}_{L^{\infty}(\Omega;\R^{d})}\norm{ B}^2_{L^{\infty}(\Omega;\R^{d\times m})}\right)\\
\ecart\dis+\norm{f}_{Lip(\Omega;\R^{d})}\norm{\nabla w}_{L^{\infty}(\Omega;\R^{d})}|x-y|+\frac{2}{\beta}\frac{|x-y|^{2}}{\lambda}\norm{B}_{Lip(\Omega);\R^{d\times M}}^2\norm{\nabla w}_{L^{\infty}(\Omega;\R^{d})}\Big\}.
\end{array}$$
Then by \eqref{MoreauEnvProp:item00:0} and \eqref{MoreauEnvProp:item0:1}  we get that $|g_{\lambda}(x)-g(x)|$ tends to 0 as $\lambda$ tends to $0^{+}$ uniformly in $\omega_{\delta}$.
These estimates will be used in the proofs of Theorems \ref{SemiConvexConvergenceTheo} and \ref{HolderConvergenceTheo} below.

\end{rem}

\section{Proof of the main results}
\label{Sec:Approx}
\label{Convergence result}

\begin{proof}[Proof of \Cref{C1ConvergenceTheo}]
Let $\tilde{T}_{\varepsilon}$ be the largest time such that $y_{\varepsilon}([0,\tilde{T}_{\varepsilon}];\omega,u_\varepsilon)\subset\omega_\delta.$ By \Cref{LyapunovStabLemma} (a) with $g=0$ we know that
    $$\tilde{T}_{\varepsilon}\geq \frac{\delta}{\norm{B(u_\varepsilon-u_{V})}_{L^{\infty}(\omega_\delta;\R^{d})}\norm{\nabla w}_{L^{\infty}(\omega_{\delta};\R^{d})}}. $$
Consequently, by \eqref{C1ConvergenceTheo:eq1} and \eqref{C1ConvergenceTheo:eq2} there exists $\varepsilon_0>0$ such that for all $\varepsilon\leq \varepsilon_0$ we have
$$\tau_{\varepsilon}\leq \frac{\delta}{\norm{B(u_\varepsilon-u_{V})}_{L^{\infty}(\omega_\delta;\R^{d})}\norm{\nabla w}_{L^{\infty}(\omega_{\delta};\R^{d})}}  $$
and thus $\tau_{\varepsilon}\leq \tilde{T}_\varepsilon$ and \Cref{hypo:scapetime} holds. Hence, we can apply \Cref{ConvTheo1} and by using \eqref{C1ConvergenceTheo:eq2} we obtain \eqref{C1ConvergenceTheo:eq3}.
\end{proof}

\begin{proof}[Proof of \Cref{SemiConvexConvergenceTheo}]
Since $\omega_{\delta}\Subset\Omega$ and $V$ is semi-convex, then by \Cref{MoreauEnvProp} there exists $\lambda_0>0$ such that $\omega_{\delta}\subset\Omega^{\lambda}_V$ and $\mathcal{M}_{\lambda}V$ is $C^{1,1}$ in $\Omega^{\lambda}_V$. Using a diagonalization argument and the properties of mollification we have get \eqref{SemiConvexConvergenceTheo:eq1}. Since $V$ is semi-convex in $\Omega$,  it is Lipschitz continuous in $\omega_{\delta}$. By \Cref{rem:lyapunov:conv} we have that $ \lim_{\lambda\to 0^+} g_\lambda=0$, where $g_\lambda$ was defined in \eqref{LyapunovStabLemma:glambda}. Together with the definitions of $
\sigma^{1}_{\varepsilon(\lambda),\lambda}$ and $\sigma^2_{\varepsilon(\lambda)}$ we obtain
\beq  \lim_{\lambda\to 0^{+}}\left(\sigma^{1}_{\varepsilon(\lambda),\lambda}+\frac{\sigma^{2}_{\varepsilon(\lambda)}}{\beta}\norm{\nabla\mathcal{M}_{\lambda}V}_{L^{\infty}(\omega_{\delta})}\right)=0.\label{SemiConvexConvergenceTheo:eq5}\eeq
Combining this with \eqref{SemiConvexConvergenceTheo:eq2}, we get $\lim_{\lambda\to 0^{+}}\tau_{\lambda}=\infty$.  Moreover, by \Cref{LyapunovStabLemma}  \eqref{LyapunovStabLemma:c}  we obtain $y([0,\tau_{\lambda}];\omega,u_\lambda)\subset\omega_{\delta}$ for $\lambda\in (0,\lambda_0)$, where $u_\lambda=u_{V_{\varepsilon(\lambda),\lambda}}$.

For the second part of \eqref{SemiConvexConvergenceTheo:eq4} we observe that $\mathcal{M}_{\lambda}V$ satisfies \eqref{rem:MoreauEnv:HJB:approxeq}. Hence, the hypotheses of \Cref{ConvTheo1} are satisfied by $V_{\varepsilon(\lambda),\lambda}$ and $\mathcal{M}_{\lambda}V$, which implies that
\beq \begin{array}{l}
     \dis\norm{\left(\mathcal{V}_{u_{\lambda},\tau_{\lambda}}(y_0)+\mathcal{M}_{\lambda}V\circ y(\tau_{\lambda};\cdot,u_{{\lambda}}) -\mathcal{M}_{\lambda}V\right)^{+}}_{L^{\infty}(\omega)} \leq \tau_{\lambda}\left(\norm{h_{\lambda}}_{L^{\infty}(\omega_{\delta})}+
     \beta\norm{u_\lambda-u_{\mathcal{M}_{\lambda}V}}^{2}_{L^{\infty}(\omega_{\delta};\R^{m})}\right),
\end{array}\label{SemiConvexConvergenceTheo:eq6}\eeq
where $h_{\lambda}$ is defined in \eqref{rem:MoreauEnv:HJB:g}. As was pointed out in \Cref{rem:MoreauEnv:HJB}, since $V$ is Lipschitz continuous in $\omega_{\delta}$, then $\norm{h_{\lambda}}_{L^{\infty}(\omega_{\delta})}\leq K\lambda $  for some constant $K$ independent of $\lambda$. Applying this estimate and \eqref{SemiConvexConvergenceTheo:eq1} in \eqref{SemiConvexConvergenceTheo:eq6}, we obtain
\beq \begin{array}{l}
     \dis\norm{\left(\mathcal{V}_{u_{\lambda},\tau_{\lambda}}(y_0)+\mathcal{M}_{\lambda}V\circ y(\tau_{\lambda};\cdot,u_{{\lambda}}) -\mathcal{M}_{\lambda}V\right)^{+}}_{L^{\infty}(\omega)}\leq \tau_{\lambda}\left(K+\frac{\norm{B}^2_{L^{\infty}(\omega_{\delta};\R^{d\times m})}}{\beta}\right)\lambda.
\end{array}\label{SemiConvexConvergenceTheo:eq7}\eeq
This together with the fact that $\mathcal{M}_{\lambda}V$ converges to $V$ in $C_{loc}(\Omega)$ and \eqref{SemiConvexConvergenceTheo:eq2} imply \eqref{SemiConvexConvergenceTheo:eq4}.
\end{proof}

\begin{proof}[Proof of \Cref{SemiConcaveConvergenceTheo}]
Let $\varepsilon_{0}>0$ be such that $\omega_{\delta}+B(0,\varepsilon)\subset\Omega$ for all $\varepsilon\in ( 0,\varepsilon_0)$. Then by the properties of mollification  we find \beq\lim_{\varepsilon\to 0^{+}}\norm{u_{\varepsilon}-u_V}_{L^{p}(\omega_{\delta};\R^{M})}=0,\label{SemiConcaveConvergenceTheo:eq1}\eeq
 where $u_\varepsilon = u_{V_\varepsilon}$,  and by Proposition \ref{prop:hessbound} we have
 \beq \begin{array}{l}
     \dis  -tr(B(y)Du_{\varepsilon}(y)) \leq \frac{1}{\beta}tr(B(y)B(y)^{\top}\nabla V_{\varepsilon}^{2})+\frac{md^2}{\beta}\norm{B}_{Lip(\omega_{\delta})}^{2}\norm{\nabla V}_{L^{\infty}(\omega_{\delta})}\\
     \ecart\dis \leq C\frac{1}{\beta}\left(md\norm{B}_{L^{\infty}(\omega_{\delta};\R^{d\times m})}^{2}+md^2\norm{B}_{Lip(\omega_{\delta};\R^{d\times m})}^{2}\norm{\nabla V}_{L^{\infty}(\omega_{\delta};\R^{d})}\right),
\end{array}\label{SemiConcaveConvergenceTheo:eq2}\eeq for all $y\in\omega_{\delta}$ and $\varepsilon\in (0,\varepsilon_0)$.
Therefore \eqref{ConvTheo2:eq0} is satisfied with $\Omega_1=\omega_\delta$ and $\bar v= V_\varepsilon$.

By \Cref{LyapunovStabLemma} (\ref{LyapunovStabLemma:b}) and \eqref{SemiConcaveConvergenceTheo:eq4} we have that $y([0,\tau_{\varepsilon}];\omega,u_{\varepsilon})\subset\omega_\delta$. Further, by \eqref{SemiConcaveConvergenceTheo:eq3}, \eqref{SemiConcaveConvergenceTheo:eq4} and the definition of $\sigma^{1}_{\varepsilon}$ and $\sigma^{2}_{\varepsilon}$ we obtain
$\lim_{\varepsilon\to 0^{+}}\tau_{\varepsilon}=\infty.$ Moreover, all the hypotheses of \Cref{ConvTheo2} are met with $  \Omega_1 = \omega_\delta$, $\bar v = u_\varepsilon= u_{V_\varepsilon}$, and $v=V$,  and thus we have
\beq \norm{\left(\mathcal{V}_{u_{\varepsilon},\tau_{\varepsilon}}(y_0)+V\circ y(\tau_{\varepsilon};\cdot,u_{\varepsilon}) -V\right)^{+}}_{L^{p}(\omega)}\leq
\tau_{\varepsilon}^{\frac{p-1}{p}}\beta\left(\frac{e^{K\tau_{\varepsilon}}-1}{K}\right)^{\frac{1}{p}} \norm{u_\varepsilon-u_{V}}^{2}_{L^{2p}(\omega_{\delta};\R^{m})}.\label{SemiConcaveConvergenceTheo:eq6}\eeq
By virtue of \eqref{SemiConcaveConvergenceTheo:eq3} and \eqref{SemiConcaveConvergenceTheo:eq4}, then \eqref{SemiConcaveConvergenceTheo:eq5} follows from \eqref{SemiConcaveConvergenceTheo:eq6}.
\end{proof}

\begin{proof}[Proof of \Cref{HolderConvergenceTheo}]
Since $\omega_\delta\Subset\Omega$ there exists $\lambda_0\in (0,1)$ such that for all $\lambda\in (0,\lambda_0)$ we have $\omega_\delta\subset\Omega_{V}^{\lambda}$ which in turns implies that there exists $\varepsilon(\lambda)$ satisfying that for all $\varepsilon\in (0,\varepsilon_0(\lambda))$ it holds that $\omega_\delta+B(0,\varepsilon)\Subset\Omega_\phi^{\lambda}$. Let us denote by $V_{\varepsilon}$ as a mollification of $V$, where  $\varepsilon>0$. We choose $\varepsilon(\lambda)$ such that  $\varepsilon(\lambda)\leq \min(\lambda^{\frac{1}{2-\alpha}}, \varepsilon_0(\lambda))$. Since $\mathcal{M}_{\lambda}V$ is semi-concave with constant $\frac{1}{\lambda}$ in $\Omega^{\lambda}_{V}$  we obtain by \Cref{MoreauEnvProp} (\ref{MoreauEnvProp:item1}), \eqref{MoreauEnvProp:item00:1}
\label{HolderConvergenceTheo:proof:semiconcaveV}, and \Cref{prop:hessbound} applied to the mollification of $\mathcal{M}_{\lambda}$
\begin{equation}
\norm{\nabla V_{\varepsilon(\lambda),\lambda}}_{L^{\infty}(\omega_{\delta})}\leq C\lambda ^{\frac{\alpha-1}{2-\alpha}} \,\text{ and }\, y^{\top} \nabla^2  V_{\varepsilon(\lambda),\lambda}(x) y \le \frac{1}{\lambda} |y|^2, \text{ for all } y \in\R^d, \,x\in \omega_\delta,
\end{equation}
where $C>0$ is a constant depending only on $V$ and $\Omega$. Applying these bounds to
 $$ -tr(B(y)Du_{\lambda}(y))=\frac{1}{\beta}\sum_{i=1}^{d}\sum_{j=1}^{m} B_{i,j}(y)\left(\sum_{k=1}^{d}\frac{\partial B_{k,j}}{\partial x_i}(y)\frac{\partial V_{\varepsilon(\lambda),\lambda}}{\partial x_k}(y)+B_{k,j}(y)\frac{\partial^2 V_{\varepsilon(\lambda),\lambda}}{\partial x_k\partial x_i}(y)\right),$$
we obtain that for $\lambda\in (0,\lambda_0)$
 \beq -tr(B(y)Du_{\lambda}(y))\leq \frac{md(d+1)}{\lambda\beta }C\norm{B}_{Lip(\omega_{\delta};\R^{d\times m})}^{2}\mbox{ for all } y\in \omega_{\delta}\label{HolderConvergenceTheo:proof:hessineq}\eeq
for a constant $C>0$ which only depends on $\Omega$ and $V$. Here $u_\lambda =u_{V_{\varepsilon(\lambda),\lambda}}$ and we  used that $\lambda^{\frac{\alpha-1}{2-\alpha}}<\frac{1}{\lambda}$ for $\lambda\in (0,1)$.

 We now prove the first part of \eqref{HolderConvergenceTheo:eq5}. For this, we need to verify that
\beq  \lim_{\lambda\to 0^{+}}\left(\sigma^{1}_{\varepsilon(\lambda),\lambda}+\frac{\sigma^{2}_{\varepsilon(\lambda)}}{\beta}\norm{\nabla V_{\varepsilon(\lambda),\lambda}}_{L^{\infty}(\omega_{\delta})}\right)=0.\label{HolderConvergenceTheo:proof:sigma1sigma2conv}\eeq
 Since $B$ is Lipschitz continuous and $\nabla w$ are Hölder continuous, and $\varepsilon(\lambda)\leq \lambda^{\frac{1}{2-\alpha}}$ we have the following estimate for $\sigma^2_{\varepsilon(\lambda)}$  $$ \sigma^2_{\varepsilon(\lambda)}\leq \norm{B}_{Lip(\omega_\delta;\R^{d\times m})}\norm{\nabla w}_{L^{\infty}(\omega_\delta;\R^d)}\lambda^{\frac{1}{2-\alpha}}+C\norm{B}_{L^{\infty}(\omega_\delta;\R^{d\times m})} \lambda^{\frac{\sigma}{2-\alpha}},$$
which implies that
\beq\sigma^2_{\varepsilon(\lambda)}\norm{\nabla V_{\varepsilon(\lambda),\lambda}}_{L^{\infty}(\omega_{\delta})}\leq C\left(\lambda ^{\frac{\sigma+\alpha-1}{2-\alpha}}+\lambda^{\frac{\alpha}{2-\alpha}}\right),\label{HolderConvergenceTheo:proof:sigma2conv}\eeq
for a constant $C>0$ which only depends on $B$, $f$, $\omega_\delta$, $w$ and $V$. Let us recall the definition of $g_\lambda$ from \eqref{LyapunovStabLemma:glambda}. For $\lambda\in (0,\lambda_0)$ we know that $\omega_\delta\subset\Omega_V^\lambda$ which combined with \eqref{rem:lyapunov:conv:eq1}
permits us to find  a constant $C>0$ depending $f$, $B$, $w$ and $\Omega$ such that
\beq g^1_{\lambda}(y)\leq C\left(\lambda^{\frac{\sigma}{2-\alpha}}+\lambda^{\frac{\sigma+\alpha-1}{2-\alpha}}+\lambda^{\frac{1}{2-\alpha}}+\lambda^{\frac{\alpha}{2-\alpha}}\right).\eeq
Using the above inequality, the Hölder continuity of $w$, the Lipschitz continuity of $B$ in $\Omega$ and the fact that $\varepsilon(\lambda)\leq \lambda^{\frac{1}{2-\alpha}}$ we obtain
\beq
 \sigma^1_{\varepsilon(\lambda)\lambda}\leq C\left(\lambda^{\frac{\sigma}{2-\alpha}}+\lambda^{\frac{\sigma+\alpha-1}{2-\alpha}}+
 \lambda^{\frac{1}{2-\alpha}}+\lambda^{\frac{\alpha}{2-\alpha}}\right),
 \label{HolderConvergenceTheo:proof:sigma1conv}
\eeq
    where $C>0$ is a constant which only depends on $f$, $B$, $w$ and $\Omega$. Combining \eqref{HolderConvergenceTheo:proof:sigma2conv} and \eqref{HolderConvergenceTheo:proof:sigma1conv} we obtain that \eqref{HolderConvergenceTheo:proof:sigma1sigma2conv} holds.
    This together with \eqref{HolderConvergenceTheo:eq2} implies that $\lim_{\lambda\to 0^{+}}\tau_{\lambda}=\infty$, see \eqref{HolderConvergenceTheo:eq4}.

 Moreover, by \eqref{HolderConvergenceTheo:eq4} and \Cref{LyapunovStabLemma}\eqref{LyapunovStabLemma:c}, we have $y([0,\tau_{\lambda}];\omega,u_{\lambda})\subset\omega_{\delta}$.

For proving the second claim in \eqref{HolderConvergenceTheo:eq5} we notice that $\mathcal{M}_{\lambda}V$ satisfies \eqref{rem:MoreauEnv:HJB:approxeq} with $v=\mathcal{M}_{\lambda}V$ in $\Omega_V^\lambda$.
Hence, the hypotheses of \Cref{ConvTheo2} are satisfied with $v=\mathcal{M}_\lambda V$ and $\overline{v}=V_{\varepsilon(\lambda),\lambda}$, which implies that
\beq \begin{array}{l}
     \dis\norm{\left(\mathcal{V}_{u_{\lambda},\tau_{\lambda}}(y_0)+\mathcal{M}_{\lambda}V\circ y(\tau_{\lambda};\cdot,u_{{\lambda}}) -\mathcal{M}_{\lambda}V\right)^{+}}_{L^{p}(\omega)} \\
     \ecart\dis \leq \tau_{\lambda}|\omega|^{\frac{1}{p}}\norm{h_{\lambda}}_{L^{\infty}(\omega_{\delta})}
     +\tau_{\lambda}^{\frac{p-1}{p}}\beta\left(\frac{e^{K(\lambda)\tau_{\lambda}}-1}{K(\lambda)}\right)^{\frac{1}{p}} \norm{u_\lambda-u_{\mathcal{M}_{\lambda}V}}^{2}_{L^{2p}(\omega_{\delta};\R^{m})},
\end{array}\label{HolderConvergenceTheo:eq7}\eeq
where $h_{\lambda}$ is defined  in  \eqref{rem:MoreauEnv:HJB:g}. As was pointed out in \Cref{rem:MoreauEnv:HJB}, for $\alpha\in \left(\frac{1}{2},1\right]$ and $\lambda$ small enough we have that $h_{\lambda}\leq C\lambda^{\frac{2\alpha-1}{2-\alpha}}$ for a constant $C>0$ depending only on $\Omega$, $\ell$, $f$, $B$ and $V$. Using this and \eqref{HolderConvergenceTheo:eq2} in  \eqref{HolderConvergenceTheo:eq7} we get
$$ \lim_{\lambda\to 0^+}\norm{\left(\mathcal{V}_{u_{\lambda},\tau_{\lambda}}(y_0)+\mathcal{M}_{\lambda}V\circ y(\tau_{\lambda};\cdot,u_{{\lambda}}) -\mathcal{M}_{\lambda}V\right)^{+}}_{L^{p}(\omega)} =0.$$
This estimate and the fact that $\mathcal{M}_{\lambda}V$ converges uniformly to $V$ on compact subsets by \Cref{MoreauEnvProp} \eqref{MoreauEnvProp:item0} imply that \eqref{HolderConvergenceTheo:eq5} holds.
\end{proof}

     \begin{proof}[Proof of \Cref{ConvTheoTraj}]
         Let us prove \eqref{ConvTheoTraj:1} and choose $T>0$. Due to the fact that $\lim_{\varepsilon\to 0^+}T_\varepsilon=\infty$ there exists $\varepsilon_0(T)$ such that $T<T_{\varepsilon}$ for all $\varepsilon\in (0,\varepsilon_0(T))$. Then for  $\varepsilon\in (0,\varepsilon_0(T))$ we have that $y([0,T];y_0,u_\varepsilon)\subset\Omega$ which by the boundedness of $\Omega$ implies that $\{\norm{y_{\varepsilon}}_{L^{\infty}((0,T);\R^{d})}\}_{\varepsilon\in(0,\varepsilon_0)}$ is bounded. Further, ${\{u_{\varepsilon}\circ y_{\varepsilon}\}}_{\varepsilon\in(0,\varepsilon_0)}$ is bounded in $L^{2}((0,T);\R^{m})$ since $\norm{u_{\varepsilon}\circ y_{\varepsilon}}^{2}_{L^{2}((0,T);\R^{m})}\leq \frac{1}{\beta}\mathcal{V}_{u_\varepsilon,T_\varepsilon}(y_0)$ and $\{\mathcal{V}_{T_{\varepsilon},u_{\varepsilon}}(y_0)\}_{\varepsilon\in(0,\varepsilon_0)}$ is bounded. since
         This, together with  the fact that $\{y_{\varepsilon}\}_{\varepsilon\in(0,\varepsilon_0)}$ is bounded in $L^{\infty}((0,T);\R^{d})$, and that $f$ is a locally Lipschitz function, implies that $\{y_\varepsilon\}_{\varepsilon\in(0,\varepsilon_0)}$ is bounded in $W^{1,2}((0,T);\R^{d})$.

         Therefore, there exist $y^{*}\in W^{1,2}((0,T);\R^{d})$ and $u^{*}\in L^{2}((0,T);\R^{m})$, and a sub-sequence of $(u_{\varepsilon}\circ y_{\varepsilon},y_{\varepsilon})$ which converges to $(u^{*},y^{*})$ weakly in $L^{2}((0,T);\R^{m})\times W^{1,2}((0,T);\R^{d})$. Given that the previous holds for each $T$, by a diagonal argument we have that $u^{*}\in L_{loc}^{2}((0,\infty);\R^{m})$ and  $y^{*}\in W_{loc}^{1,2}((0,\infty);\R^{d})$, and a sub-sequence of  $(u_{\varepsilon}\circ y_{\varepsilon},y_{\varepsilon})$ converges to $(u^{*},y^{*})$ weakly in $L_{loc}^{2}((0,\infty);\R^{m})\times W_{loc}^{1,2}((0,T);\R^{d}).$ This proves the existence of at least one weakly convergent sub-sequence.

     We now prove \eqref{ConvTheoTraj:2}, i.e., the optimality for every accumulation point of $(u_{\varepsilon}\circ y_{\varepsilon},y_{\varepsilon})$. Consider an arbitrary pair $(u,y)\in L_{loc}^{2}((0,\infty);\R^{m})\times W_{loc}^{1,2}((0,T);\R^{d})$ and a sub-sequence of $(u_{\varepsilon}\circ y_{\varepsilon},y_{\varepsilon})$, still denoted by $(u_{\varepsilon}\circ y_{\varepsilon},y_{\varepsilon})$, such that $(u_{\varepsilon}\circ y_{\varepsilon},y_{\varepsilon})$ converges to $(u,y)$ weakly in $L_{loc}^{2}((0,\infty);\R^{m})\times W_{loc}^{1,2}((0,\infty);\R^{d}).$ Then, by the compact injection of $W^{1,2}_{loc}((0,\infty);\R^{d})$ in $C_{loc}([0,\infty);\R^{d})$ passing to a sub-sequence we have that $y_{\varepsilon}$ converges to $y$ in $C_{loc}([0,\infty);\R^{d})$. In particular, this implies that $y$ is a solution of \eqref{OpenLoop}. Moreover, due to the continuity of $\ell$, the $L^{2}((0,T);\R^{m})$ norm lower-semi continuity for all $T>0$, and the non-negativity of $V$  we have the inequality in the following statements, the equality follows from assumption \eqref{eq:aux5}:
         $$ \int_{0}^{T}\left(\ell(y)+\frac{\beta}{2}|u|^{2}\right) dt\leq \lim_{\varepsilon \to 0^{+}}\mathcal{V}_{u_\varepsilon, T_\varepsilon}(y_{0})+V(y_{\varepsilon}(T_{\varepsilon}))= V(y_0).$$
        Since this holds for every $T>0$, we obtain that $u\in L^{2}((0,\infty);\R^{m})$ and $J(u)\leq V(y_0)$. Hence, $u^*$ is an optimal solution of \eqref{ControlProblem}.
         Since this is true for any accumulation point of $(u_\varepsilon,y_\varepsilon)$, then \eqref{ConvTheoTraj:cons:eq:1} holds.

         To prove \eqref{ConvTheoTraj:3} we note that if the the solution of \eqref{ControlProblem} is unique, then by \eqref{ConvTheoTraj:2} every sub-sequence of $u_{\varepsilon}\circ y_{\varepsilon}$ has a convergent sub-sequence which converges to the optimal solution, since this solution is unique, the whole sequence $u_{\varepsilon}\circ y_{\varepsilon}$ converges to the optimal solution.\end{proof}

        \begin{proof}[Proof of \Cref{Corconv}]  By \eqref{ConvTheoTraj:eq:1} there exists a sub-sequence ${\mathcal{V}}_{u_\varepsilon, T_\varepsilon}+V\circ y({T_\varepsilon};\cdot,u_{\varepsilon})$ which converges almost everywhere in $\omega$, \ and thus  \eqref{eq:aux5} holds almost everywhere in $\omega$, and (a)-(c) of Theorem \ref{ConvTheoTraj} hold.  Furthermore, if $p=\infty$, using the continuity of $V$ and $\mathcal{V}_{u_\varepsilon, T_\varepsilon}$  the whole sequence converges everywhere in $\omega$ and therefore  \eqref{ConvTheoTraj:1}, \eqref{ConvTheoTraj:2} and \eqref{ConvTheoTraj:3} hold for every $y_0\in\omega$.
     \end{proof}

\section{Example}
\label{Example}
In this section we provide an example of a control problem of the form \eqref{ControlProblem} with a locally Lipschitz continuous value function which satisfies   \Cref{Lyapunov} for a smooth function $w$ with $g=0$. Additionally, we identify a set of initial conditions where the optimal control of \eqref{ControlProblem} is not unique, which will imply that the value function is not differentiable in that set.

Let us start by defining the running cost for the state variable by
\begin{equation*}
\ell_\alpha(y) = \frac{1}{2}|y|^2\left(1+\alpha\psi\left(\frac{|y-z|}{\sigma}\right)\right),
\end{equation*}
where
 $$\psi(s)=\left\{\begin{array}{ll}
    \exp\left(-\frac{1}{1-s^2}\right) & \mbox{ if }|s|<1 \\
     0& \mbox{ if }|s|\geq 1
\end{array}\right. ,$$
and  $z\in \R^{2}$ satisfies $z_{1}< -\sigma$, $z_2=0$,  for some $\sigma >0$.
For $\alpha\in [0,\infty)$, we consider the  following control problem
\beq \min_{\begin{array}{c}
     u\in L^{2}((0,\infty);\R^{2}),  \\
      y'=u,\ y(0)=y_0
\end{array}} \int_{0}^{\infty} \ell_\alpha(y(t))dt+\frac{\beta}{2}\int_{0}^{\infty}|u(t)|^2 dt.
\label{Example:controlproblem}\eeq
 The value function of this problem is denoted  by $V_\alpha$. There are two cases of interest. They are $\alpha=0$ and $\alpha$ tending to infinity. In the first case the value function is a quadratic function given by $V_{0}(y)=\frac{\sqrt{\beta}}{2}|y|^2$ and it therefore is of class $C^{\infty}(\R^{2})$.
 In the second case we are going to prove that as $\alpha$ tends to infinity the value function tends to the value function $V_\infty$ of the following state constrained problem
\beq \min_{\begin{array}{c}
     u\in L^{2}((0,\infty);\R^{2}),  \\
      y'=u,\ y(0)=y_0, \\
      y(t)\in \R^{2}\setminus B(z,\sigma) \mbox{ for all }t\geq 0.
\end{array}} \int_{0}^{\infty} \frac{1}{2}|y(t)|^2dt+\frac{\beta}{2}\int_{0}^{\infty}|u(t)|^2 dt.
\label{Example:controlproblem:infinity}\eeq
on compact subsets of $\R^{2}\setminus B(z,\sigma)$. Here we set $B(z,\sigma)=\{u \in \R^2: |u-z|< \sigma \}$.  Accordingly we are going to prove  that $V_{\alpha}$ is not smooth for $\alpha$ large enough. This will be due to the lack of uniqueness of solutions for $y_0$ in $\{(y_1,y_2)\in\R^{2}:\ y_2=0,\ y_1<z_1-\sigma\}$ in problem \eqref{Example:controlproblem:infinity}.

\begin{prop}
\label{Example:propOptimalityCond}
    Let $\alpha\in (0,\infty)$. For each $y_0\in\R^{2}$ the control problem \eqref{Example:controlproblem} has an optimal solution and the associated value function is bounded as follows
    \beq   \frac{\sqrt{\beta}}{2}|y_0|^2\leq V_{\alpha}(y_0)\leq \frac{\sqrt{\beta}}{2}\left(1+\frac{\alpha}{2}\right)|y_0|^2.\label{Example:VB}\eeq
    Further, for every $y_0\in \R^{2}$ and every  optimal control $u^{*}\in L^{2}((0,\infty);\R^2)$ of \eqref{Example:controlproblem}  with corresponding trajectory $y^{*}$  it holds that $u^{*},y^{*}\in C^{\infty}([0,\infty);\R^{2})$, and
    \beq \frac{du^*}{dt}(t)=\frac{d^2y^*}{dt^2}(t)=\frac{1}{\beta}\nabla \ell_{\alpha}(y^*(t))\mbox{ for }t>0,\label{Example:OptimalityCond}\eeq
    \beq \ell_{\alpha}(y^{*}(t))=\frac{\beta}{2}|u^{*}(t)|^2\mbox{ for all }t>0\label{Example:trajHJB}\eeq
    and in particular
    \beq |u^*(0)|\leq \left(\frac{1+\alpha}{\beta}\right)^{\frac{1}{2}}|y_0|.\label{Example:LipscthizCond:u0bound}\eeq
\end{prop}
\begin{proof}
Since the control system appearing in \eqref{Example:controlproblem} is controllable  the existence of an optimal control-state pair $(u^*,y^*$ can easily be established. By the Lagrange multiplier's theorem there exists $p^*\in L^2((0,\infty);\R^2)$ satisfying
\begin{equation}\label{eq:optsyst}
 p^*(t)'=\nabla \ell_\alpha (y^*(t))\mbox{ and }u^*=\frac{1}{\beta}p^*(t).
\end{equation}
Utilizing this first order optimality system it is straightforward to see that the optimal control of \eqref{Example:controlproblem} for $\alpha=0$ and $y_0\in\R^{2}$ is given by $u^{0}=-\frac{1}{\sqrt{\beta}}y_0e^{-\frac{t}{\sqrt{\beta}}}$. The trajectory associated to this control is given by $y^0=y_0e^{-\frac{t}{\sqrt{\beta}}}$. Turning to the case $\alpha>0$ we observe that $\psi \leq (1+\alpha)$, and thus  $u^{0}$ is feasible for \eqref{Example:controlproblem} with $\alpha>0$. We obtain that $V_{\alpha}$ is bounded from above in the following manner
    \beq V_\alpha(y_0)\leq  \frac{\sqrt{\beta}}{2}\left(1+\frac{\alpha}{2} \right)|y_0|^2\mbox{ for all }y_0\in \R^{d}.\label{Example:VUB}\eeq
Furthermore, since every feasible control for \eqref{Example:controlproblem} with $\alpha>0$ is feasible for $\alpha=0$ and since $\psi$ is bounded from below by zero, we arrive at
\beq \frac{\sqrt{\beta}}{2}|y_0|^2=V_{0}(y_0)\leq V_\alpha(y_0) \mbox{ for all }y_0\in\R^{d}.\label{Example:VLB}\eeq
    This proves \eqref{Example:VB}. In order to continue the analysis of this problem, we need the first order optimality condition of \eqref{Example:controlproblem}.

From \eqref{eq:optsyst} we now deduce that

\begin{equation}\label{eq:aux1} \frac{d^{2}y^{*}}{dt}(t)=\frac{du^*}{dt}(t)=\frac{1}{\beta}\nabla \ell_\alpha (y^{*}(t)),
\end{equation}
 which in turns proves \eqref{Example:OptimalityCond}. We further obtain that $y\in H^2((0,\infty);\R^2),\, u^*\in H^1((0,\infty);\R^2),\, p^*\in H^1((0,\infty);\R^2)$, from which $\lim_{t\to \infty} y^*(t)=\lim_{t\to \infty} u^*(t)=\lim_{t\to \infty} p^*(t)=0$ follows.
By iteration and the regularity of $\psi$ we  obtain that $u^{*}$ and $y^{*}$ are elements of $C^{\infty}([0,\infty);\R^{2})$. Multiplying the right hand side of \eqref{eq:aux1} by $u^{*}(t)$, using that $u^{*}(t)=\frac{d}{dt}y^{*}(t)$, and integrating from $t$ to infinity we obtain \eqref{Example:trajHJB}. Using \eqref{Example:trajHJB}, the continuity of $u^*$ and $y^*$ at 0, and the definition of $\ell_{\alpha}$ we get \eqref{Example:LipscthizCond:u0bound}.
\end{proof}

\begin{lemma}
\label{Example:lemma:Stability}
If for $y_0\in \R^{2}\setminus \overline{B(z,\sigma)}$ and $u_{0}\in \R^{2}$  the solution $y\in C^{\infty}([0,\infty),\R^{2})$ of
$$ y''(t)=\frac{1}{\beta}\nabla \ell_{\alpha}(y)\mbox{ for all }t>0 \mbox{ and }y(0)=y_0,\ y'(0)=u_0$$
satisfies $\lim_{t\to \infty}y(t)=0$, then $u_0\cdot y_0<0$.
\end{lemma}
\begin{proof}
For $y_0,u_0$ as in the statement of the lemma, we shall prove that  $u_0\cdot y_0<0$. Using the continuity of $y$ and that  $y_0\notin \overline{B(z,\sigma)}$ there exists $T>0$ such that $y(t)\notin\overline{B(z,\sigma)}$ for $t\in [0,T]$. Let $\hat{T}$ be the largest $T$ such that $y(t)\notin\overline{B(z,\sigma)}$ for $t\in [0,\hat{T}]$. Then  $y$ satisfies
$$ y''(t)=\frac{1}{\beta}y(t) \text{ for }t\in (0,\hat{T}).$$
Therefore we obtain the following expression for $y(t)$ with $t\in (0,\hat{T})$
\beq  y(t)=y_0\cosh\left(\frac{t}{\sqrt{\beta}}\right)+u_0\sqrt{\beta}\sinh\left(\frac{t}{\sqrt{\beta}}\right).
\label{Example:Stability:linsol}\eeq
Differentiating this expression we have
\beq  y'(t)=\frac{1}{\sqrt{\beta}}y_0\sinh\left(\frac{t}{\sqrt{\beta}}\right)+u_0\cosh\left(\frac{t}{\sqrt{\beta}}\right).\label{Example:Stability:linsol_der}\eeq
Multiplying $y(t)$ and $y'(t)$ for $t\in (0,\hat{T})$ we arrive at
$$ \frac{d}{dt} \frac{1}{2} |y(t)|^2= y(t)\cdot y'(t)=\sinh\left(\frac{t}{\beta}\right)\cosh\left(\frac{t}{\beta}\right)\left(\frac{1}{\sqrt{\beta}}|y_0|^2+\sqrt{\beta}|u_0|^2\right)+y_0\cdot u_0\cosh\left(2\frac{t}{\beta}\right).$$
If $y_0\cdot u_0\geq 0$ then the previous inequality
implies that $y'(t)\cdot y(t)\geq 0$ which in turns allow us to deduce that $|y(t)|^2\geq |y_0|^{2}$ for all $t\in [0,\hat{T}]$. If $\hat{T}<\infty$ then by the continuity of $y$ and since $y_0\notin \overline{B(z,\sigma)}$ we know that there exists $\delta>0$ such that $y(t)\notin \overline{B(z,\sigma)}$ for all $t\in (\hat{T},\hat{T}+\delta)$ which contradicts the definition of $\hat{T}$. Therefore $\hat{T}=\infty$ which contradicts the fact that $\lim_{t\to\infty}y(t)=0$. This permits us to conclude that $u_0\cdot y_0<0$.
\end{proof}

In the following we present the necessary optimality conditions arising from the dynamic programming principle which are important for proving  the non-differentiability of the value function. We first prove that in the ball $B(0,|z|-\sigma)$ the value function coincides with $V_0$ for every $\alpha>0$. This together with the stability of the optimal trajectories will allow us to get a necessary optimality condition which involves $D^+V_{\alpha}(y_0)$. This optimality condition also implies that $D^+V_{\alpha}(y_0)$ is bounded for every $y_0$. Combining this with some technical results we shall prove that $V_{\alpha}$ is Lipschitz continuous.
\begin{lemma}
    \label{Example:DifZero}
    Let $y_0\in B(0,|z|-\sigma)$ and $\alpha>0$. Then $V_{\alpha}(y_0)=V_0(y_0)=\frac{\sqrt\beta}{2}|y_0|^2$.
\end{lemma}
\begin{proof}
    By \eqref{Example:OptimalityCond} and \Cref{Example:lemma:Stability} we have that $\frac{d}{dt} \frac{1}{2}|y^*(t)|^2= y^{*}(t)\cdot u^{*}(t)< 0$ for all $t\geq 0$. In particular we have that $y^{*}(t)\in B(0,|z|-\sigma)$ for all $t>0$. This implies that
    $$ \int_{0}^{\infty}\ell_{\alpha}(y^{*}(t))dt=\int_{0}^{\infty}\frac{1}{2}|y^{*}(t)|^2 dt $$
    which permits us to conclude that $u^{*}$ is an optimal solution of \eqref{Example:controlproblem} with $\alpha=0$ and therefore $V_{\alpha}(y_0)=V(y_0)$.
\end{proof}

\begin{prop}
\label{Example:prop:Lipschitz}
For every $\alpha>0$ the value function of problem \eqref{Example:controlproblem} is Lipschitz continuous on compact subsets of $\R^2$.
\end{prop}
\begin{proof}
    Let $y_0\in \R^2$, $y_1\in \R^2$, and $T>0$ arbitrary. Consider $u(t)=\frac{1}{T}(y_0-y_1)$ and $y(t)=\frac{t}{T}\cdot y_0+(1-\frac{t}{T})\cdot y_1$ for $t\in [0,T]$. We notice that $y'(t)=u$, $y(0)=y_1$ and $y(T)=y_0$. Then by the dynamic programming principle we have
    $$ V_\alpha(y_1)-V_\alpha(y_0)\leq \int_{0}^{T}\left(\ell_{\alpha}(y(t))+\frac{\beta}{2}|u(t)|^2\right)dt\leq \int_{0}^{T}\left(\frac{1+\alpha}{2}\left|\frac{t}{T}(y_0-y_1)+y_1\right|^2+\frac{\beta}{2T^2}|y_0-y_1|^2\right)dt.$$
    By direct calculations on the right-hand side of the previous expression we get
    $$ V_\alpha(y_1)-V_\alpha(y_0)\leq \frac{1+\alpha}{2}\left(2T|y_1|^2+T|y_0-y_1|^2 \right)+\frac{\beta}{2T}|y_0-y_1|^2.$$
    Since this is for an arbitrary $T$ we can choose $T=|y_0-y_1|$ in the previous expression to conclude the Lipschitz continuity of $V$ on compact subsets of $\R^2$.
\end{proof}
\begin{lemma}\label{Example:ExpStab}
    Let $\alpha\in (0,\infty)$ and $y_0\in\R^2$. Consider $u^{*}$ a solution of \eqref{Example:controlproblem} and $y^*$ its associated  trajectory. Then  $y^{*}$ is exponentially asymptotically stable:
    \beq
    |y^{*}(t)|\leq (1+\alpha)^{1/2}\exp\left(\frac{-1}{(1+\alpha)\sqrt{\beta}}t\right)|y_0|\mbox{ for all }t>0.
    \label{Example:ExpStab:expin}
    \eeq
\end{lemma}
\begin{proof}
Due to the local Lipschitz continuity of $V_\alpha$ and the dynamic programming principle we have that $V_\alpha\circ y^*$ is differentiable for almost all $t>0$ and satisfies
\beq \frac{d}{d t}V_\alpha\circ y^*(t)+\ell_\alpha (y(t))+\frac{\beta}{2}|u^*(t)|^2=0 \mbox{ for almost all }t>0.\label{Example:ExpStab:proof:hjb}\eeq
Using \eqref{Example:trajHJB} in \eqref{Example:ExpStab:proof:hjb} we obtain
$$ \frac{d}{d t}V_\alpha\circ y^*(t)+2\ell_\alpha (y(t))=0.$$
Since $\ell_\alpha(y)\geq \frac{1}{2}|y|^{2}$ for all $y\in\R^2$ and using \eqref{Example:VB} in the previous equality, we get \eqref{Example:ExpStab:expin} which proves the asymptotic exponential stability of $y^{*}$.
\end{proof}

Below $\overline{co}\{\omega\}$ denotes the convex closure of a set $\omega$  in $\R^{2}$. Further for $v\in C(\omega)$ with $\omega\subset\R^{2}$ open, the set valued function $D^{+}v$ stands for  the super-differential of $v$ (see \cite[Chaper 2, Section 1]{Bardi1997}).

\begin{prop}
\label{Example:NecessaryDynamiProgCond}
    Let $y_0\in\R^{2}$, $\alpha>0$ and set $\mathcal{I}(y_0)=\overline{co}\{-\beta u^{*}(0):\ u^{*}\mbox{ is an optimal solution of }\eqref{Example:controlproblem}\}$.
     Then we have $D^{+}V_{\alpha}(y_0)=\mathcal{I}(y_0)$.
\end{prop}
\begin{rem}
     For our purpose, the relevance of  $D^{+}V_{\alpha}(y)$ relies on the fact that if $V_{\alpha}$ is differentiable at $y$, then $D^{+}V_{\alpha}(y_0)=\{\nabla V_{\alpha}(y_0)\}$ (see Lemma 1.8 in \cite[Chaper 2, Section 1]{Bardi1997}).
\end{rem}
\begin{proof}
Step 1.
Let $R>0$ be fixed. By \Cref{Example:propOptimalityCond} we know that there exists a constant $C_R>0$ such that for all $y_0\in \overline{B}(0,R)$ and all solutions $u^{*}\in L^{2}((0,\infty);\R^2)$ of \eqref{Example:controlproblem} we have
$$ \norm{u^*}_{H^1((0,\infty);\R^2)}\leq C_R.$$
Moreover, since $\overline{B}(0,R)$ is compact, \Cref{Example:ExpStab} implies  the existence of $T_R>0$ such that $y^{*}(t)\in B(0,|z|-\sigma)$ for all $t\geq T_{R}$,
where  $y^{*}$ is an optimal trajectory of \eqref{Example:controlproblem} with $y_0 \in \overline{B}(0,R)$.
Combining these statements, the dynamic programming principle, and \Cref{Example:DifZero}, we have that for all $y_0\in \overline{B}(0,R)$ the value function $V_{\alpha}$ is equal to
\begin{equation} \label{eq:aux2}
\min_{\begin{array}{c}
     u\in H^{1}((0,T_{R});\R^2),\  \norm{u}_{H^1((0,T_R);\R^2)}\leq C_R,  \\
     y'=u\mbox{ in }(0,T),\ y(0)=y_0,
\end{array}}\int_{0}^{T_R}\ell_{\alpha}(y(t))dt+\frac{\beta}{2}\int_{0}^{T_{R}}|u(t)|^2 dt+V_0(y(T_R)).
\end{equation}
Moreover there is a one-to-one relationship between solutions to \eqref{Example:controlproblem} and \eqref{eq:aux2}  by proper restriction of optimal controls of \eqref{Example:controlproblem} to $[0,T_R]$ respectively by extension.
Utilizing the fact that the optimal controls are different from zero inside $B(0,|z|-\sigma)$ it can be argued that for the optimal controls of \eqref{eq:aux2} the constraint is not active.  By the first order optimality conditions for \eqref{eq:aux2} we obtain that
$\sqrt\beta u^*(T_R)=-y^{*}(T_R)$ for each optimal control-state pair.

Step 2. To call upon a general result on the sensitivity of marginal  functions below, we now endow the set of feasible controls $\overline B_{C_R} = \{u\in H^{1}((0,T_{R});\R^2):  \norm{u}_{H^1((0,T_R);\R^2)}\leq C_R, \}$  of \eqref{eq:aux2} by the weak topology induced by $H^{1}((0,T_{R});\R^2)$.
 For $u\in H^{1}((0,T_{R});\R^2)$ and $y_0\in \overline{B}(0,R)$ we define the reduced cost-functional associated to \eqref{eq:aux2} by $J_{R}(y_0,u)=\int_{0}^{T_R}\ell_{\alpha}(y(t))dt+\frac{\beta}{2}\int_{0}^{T_{R}}|u(t)|^2 dt+V_0(y(T_R))$ where $y$ is the unique solution of $y'=u$ and $y(0)=y_0$.
 In the following we argue the continuity of $J_R:\R^2\times\overline B_{C_R} \to \R$, as well as the differentiability  with respect to $y_0$ and the continuity of the gradient considered on $\R^2\times\overline B_{C_R}$. For the continuity of $J_{R}$ let us consider $y_0^{n}$ converging to $y_0$ and $u_n$ converging to $u$ in the weak topology of $H^{1}((0,T_R);\R^2)$ with $\norm{u_n}_{H^1((0,T_R);\R^2)}\leq C_R$. By the compact embedding of $H^1((0,T_R);\R^2)$ in $C([0,T_R];\R^2)$ we have that passing to a sub-sequence $u_n$ converges to $u$ in $C([0,T_R];\R^2)$. This and the convergence of $y_0^n$ to $y_0$ implies that for  the same sub-sequence the states $y(u_n)$ converge to $y$ in  $C([0,T_R];\R^2)$, and that $y'=u$ and $y(0)=y_0$. Due to the fact that $\ell_{\alpha}$ is continuous, we get that $J_{R}(y_0^n,u_n)$ converges to $J_R(y_0,u)$ through this sub-sequence. Since this holds for each convergent sub-sequence we get that the whole sequence converges and hence $J_{R}$ is continuous. The existence of the gradient of $J_{R}$ with respect to $y_0$ is a direct consequence of the classic ODE theory and it is given by
\begin{equation}\label{eq:aux3} \nabla_{y_0} J_{R}(y_0,u)=\int_{0}^{T_{R}}\nabla \ell_\alpha(y(t))dt+\sqrt{\beta}y(T_R).
\end{equation}
The continuity of $\nabla_{y_0} J_{R}(y_0,u)$ with respect to $(y_0,u)$ can be argued similarly as the
continuity of $J_R$ by using the compact embedding of $H^1((0,T_R);\R^2)$ in $C([0,T_R];\R^2)$.

Step 3.  We are now in a position to  apply Proposition 4.4, and hence Proposition 2.1  in \cite[Chapter 2]{Bardi1997}, to the value function associated to \eqref{eq:aux2}, and hence to $V_\alpha$,   for each $y_0\in \overline B(0,R)$. This asserts that
$$
D^+V_\alpha(y_0) = \overline{co}(\{\nabla_{y_0} J_{R}(y_0,u^*):\ u^{*}\mbox{ is an optimal solution of }\eqref{Example:controlproblem}\}).
$$
Combining this fact, together with \eqref{eq:aux3}, \eqref{Example:OptimalityCond}, and $\sqrt\beta u^*(T_R)=-y^{*}(T_R)$ from Step 1, we obtain
$D^{+}V_{\alpha}(y_0)=\mathcal{I}(y_0)$.  Since $R>0$ was chosen arbitrarily, this equality holds for all $y_0 \in \R^2$.
\end{proof}

\begin{rem}
As a consequence of \eqref{eq:optsyst} the assertion of Proposition \ref{Example:NecessaryDynamiProgCond} can equivalently be expressed as $D^{+}V_{\alpha}(y_0)=\overline{co}\{-p^{*}(0):\ p^{*}\mbox{ is the adjoint state associated to  an optimal solution $u^*$ of }\eqref{Example:controlproblem}\}$.
\end{rem}

\begin{lemma}
\label{Example:lemma:Lyapunov}
For every $\alpha\in (0,\infty)$  we have $y\cdot \nabla V_{\alpha}(y)> 0$ for almost all $y\in \R^{2}\setminus \overline{B(z,\sigma)}$. Further, define $w:\R^{2}\to\R$  by
\beq w(y)=\left\{\begin{array}{ll}
    0 & \mbox{ if }|y|\leq R \\
    (|y|^2-R^2)^2 & \mbox{ if }|y|> R
\end{array}\right.\label{Example:lemma:Lyapunov:func}\eeq
with $R=|z|+\sigma$. Then $w$ satisfy \Cref{Lyapunov} for $\phi=V_{\alpha}$, with g=0,  and $\Omega$ any open bounded subset of $\R^{2}$ containing the ball $B(0,R)$ and $\omega\Subset\Omega$ an open set such that $B(0,R)\subset\omega$.
\end{lemma}
\begin{proof}
Let $y_0\in \R^{2}\setminus \overline{B(z,\sigma)}$ be such that $V_{\alpha}$ is differentiable at $y_0$. Then $D^{+}V_{\alpha}(y_0)=\{\nabla V_{\alpha}(y_0)\}$. This together with \Cref{Example:NecessaryDynamiProgCond} implies  that $u^{*}(0)=-\frac{1}{\beta}\nabla V_{\alpha}(y_0)$.  By Lemma \ref{Example:lemma:Stability} and \eqref{Example:OptimalityCond} we obtain that $u^{*}(0)\cdot y_0<0$. Therefore we have $\nabla V_{\alpha}(y_0)\cdot y_0>0$. This along with the fact that $V_{\alpha}$ is differentiable almost everywhere allow us to conclude that $\nabla V_{\alpha}(y)\cdot y>0$ for almost every $y\in \R^{2}\setminus \overline{B(z,\sigma)} $.

For the rest of the statement it is enough to notice that $\nabla w(y)=0$ in $\overline{B(0,R)}$ and $\nabla V_{\alpha}(y)\nabla w(y)>0$ for all $y\in \R^{2}\setminus \overline{B(0,R)}$ since $B(z,\sigma)\subset \overline{B(0,R)}$.
\end{proof}

Next we turn to proving the non-differentiability of $V_\alpha$. For this purpose we shall establish that there exists $\alpha^{*}$ such that for all $\alpha>\alpha^*$ there exists  $\hat y_{1,\alpha}\in (-\infty,z_1-\sigma)$ such that for each initial condition of the form $(y_1, 0)$ with $\hat y_1\leq y_{1, \alpha}$ there exists at least two optimal solutions of \eqref{Example:controlproblem}.  In the following we denote $e_{1}=(1,0)$ and $e_2=(0,1)$.
\begin{lemma}\label{Example:solucionEje}
    Let $y_{0,1}\in (-\infty,z_1-\sigma)$ and $\alpha\in (0,\infty)$. Assume that  $y^{*}\in H^{1}((0,\infty);\R^2)$  is an optimal trajectory of \eqref{Example:controlproblem} with $y_0=(y_{0,1},0)$. If $\frac{d}{dt}y^{*}(0)\cdot e_2=0$, then $y^{*}(t)=(y_1^{*}(t),0)$ for all $t>0$, with $y_1^{*}$ the unique solution of
    \beq  y_1'= -y_1\frac{1}{\sqrt{\beta}}\sqrt{\left(1+\alpha\psi\left( \frac{|y_1-z_1|}{\sigma}\right)\right)}, \, y_1(0)=y_{0,1}, \mbox{ for all }t>0,\label{Example:solucionEje:ode}\eeq
     and $y_1^{*}(t)<0$ for all $t\in (0,\infty)$.
\end{lemma}
\begin{proof}
    Since $y^*$ satisfies \eqref{Example:OptimalityCond}, we know that $y^{*}\in C^{1}([0,\infty);\R^{2})$. Moreover, since $y^{*}\in H^{1}((0,\infty);\R^2)$ we can apply Lemma \ref{Example:lemma:Stability} to deduce that  $\frac{dy^{*}}{dt}(0)\cdot e_1> 0$. Combining this, \eqref{Example:trajHJB} and the continuity of $\frac{d}{dt}y^*$ at $0$ we deduce that $\frac{dy^{*}}{dt}(0)\cdot e_1=-y_{0,1}\frac{1}{\sqrt{\beta}}\sqrt{\left(1+\alpha\psi\left( \frac{|y_{0,1}-z_1|}{\sigma}\right)\right)}$. Consider $\hat{y}(t)=(\hat{y}_1(t),0)$ for $t>0$ with $\hat{y}_1(t)$
     the unique solution of \eqref{Example:solucionEje:ode}. Differentiating $\hat{y}$ twice and using \eqref{Example:solucionEje:ode} we get
     $$ \hat{y}''(t)=\frac{1}{\beta}\nabla \ell_{\alpha}(\hat{y}),\mbox{ for all }t>0, \ \hat{y}(0)=y_0, \ \hat{y}'(0)=\left(-y_{0,1}\frac{1}{\sqrt{\beta}}\sqrt{\left(1+\alpha\psi\left( \frac{|y_{0,1}-z_1|}{\sigma}\right)\right)},0\right)=\frac{dy^{*}}{dt}(0).$$
     Thus, by uniqueness $\hat{y}=y^{*}$ which concludes the proof.
\end{proof}
\begin{lemma}
\label{Example:lemma:nondif}
    There exists $\bar{\alpha}>0$ such that for each $\alpha>\bar{\alpha}$ there exists $\hat y_{\alpha,1} \in (-\infty,z_1-\sigma)$ such that for all $y_1<\hat y_{\alpha,1}$, each optimal trajectory $y^{*}$ associated to an optimal solution $u^*$ of \eqref{Example:controlproblem} with $y^*(0):=y_0=(y_1,0)$, there exists $t^{*}\in (0,\infty)$ fulfilling $y^{*}(t^*)\cdot e_2\neq 0$.
\end{lemma}
\begin{proof}
    By contradiction, let us assume  that the lemma does not hold. Then there exist sequences $\alpha_n>0$ and $y_{n,1}<z_1-\sigma$ satisfying $y_{n+1,1}\le y_{n,1}$ and
    $$ \lim _{n\to \infty}\alpha_n=\infty, \ \lim _{n\to \infty}y_{n,1}=-\infty$$
    and there exist optimal trajectories $y_n^{*}$ of \eqref{Example:controlproblem} with $y^*_n(0)=(y_{n,1},0)$ and $\alpha=\alpha_n$ such that $y_n^{*}(t)\cdot e_2=0$ for all $t>0$.  Since $y_{n,1}$ is monotone and diverges to minus infinity, we obtain that every $y_n^{*}$ coincides with $y_{0,1}$  at a time $t_{n}\in (0,\infty)$. Defining $\tilde{y}_{n}(t)=y_n^{*}(t+t_n), \tilde{u}_{n}(t)=u_n^{*}(t+t_n)$ we have by the dynamic programming principle that $(\tilde{y}_{n}, \tilde{u}_{n})$ is an optimal trajectory-control pair   of \eqref{Example:controlproblem} with  $\alpha=\alpha_n$ and $y_0=(y_{0,1},0)$. Let $u_\infty$  be an optimal solution of \eqref{Example:controlproblem:infinity}  with $y_\infty$ its associated trajectory. By the optimality of $\tilde{u}_{n}$ and the fact that $y_{\infty}(t)\notin B(z,\sigma)$ we have
    $$ \int_{0}^{\infty}\ell_{\alpha_n}(\tilde{y}_{n}(t))dt+\frac{\beta}{2}\int_{0}^{\infty}|\tilde{u}_{n}(t)|^2 dt \leq \int_{0}^{\infty}\frac{1}{2}|y_{\infty}(t)|^{2}dt+\frac{\beta}{2}\int_{0}^{\infty}|u_\infty(t)|^2 dt$$
    Since $\ell_{\alpha_n}(y)\geq \frac{1}{2}|y|^2$ for all $y\in\R^{2}$ and $\tilde{y}_n'=\tilde{u}_n$, the former inequality implies that $\{(\tilde{y}_n,  \tilde{u}_n)\}_{n=1}^\infty$ is  a bounded family in  $H^{1}((0,\infty);\R^{2})\times L^{2}((0,\infty);\R^{2})$. Passing to a sub-sequence if necessary, there exists $\tilde{u}\in L^{2}((0,\infty);\R^{2})$ and $\tilde{y}\in H^{1}((0,\infty);\R^{2})$ such that
    $$ \tilde{u}_n\weak \tilde{u} \mbox{ in }L^{2}((0,\infty);\R^{2}),\ \tilde{y}_n\weak \tilde{y} \mbox{ in }H^{1}((0,\infty);\R^{2}),$$
    and for all $T\in (0,\infty)$
    $$\tilde{y}_n\to \tilde{y} \mbox{ in }C([0,T];\R^{2}).$$
    By the lower semi-continuity of the $L^{2}((0,\infty);\R^{2})$ norm with respect to the weak topology of  $L^{2}((0,\infty);\R^{2})$ we get $$ \int_{0}^{\infty}\frac{1}{2}|\tilde{y}(t)|^{2}dt+\frac{\beta}{2}\int_{0}^{\infty}|\tilde{u
    }(t)|^2 dt \leq \int_{0}^{\infty}\frac{1}{2}|y_{\infty}(t)|^{2}dt+\frac{\beta}{2}\int_{0}^{\infty}|u_\infty(t)|^2 dt.$$
    Moreover by the definition of $\ell_\alpha$ we get for all $T\in (0,\infty)$
    \begin{equation*}\frac{1}{2}\int_{0}^{T}|\tilde{y}_n(t)|^2\psi\left(\frac{|\tilde y_n(t)-z|}{\sigma}\right) dt\leq\frac{1}{\alpha_n}\left(\int_{0}^{\infty}\frac{1}{2}|y_{\infty}(t)|^{2}dt+\frac{\beta}{2}\int_{0}^{\infty}|u_\infty(t)|^2 dt\right).\end{equation*}
      Utilizing the uniform convergence of $\tilde{y}_n$ one obtains that  for all $T\in (0,\infty)$ one can take the limit as $n$ goes to infinity to obtain
     $$ \int_{0}^{T}|\tilde{y}(t)|^2\psi\left(\frac{|\tilde{y}(t)-z|}{\sigma}\right) dt\leq 0,$$
    from where we deduce that $\psi\left(\frac{\tilde{y}(t)-z}{\sigma}\right) =0$ for almost all $t\in (0,T)$. Since the former holds for all $T>0$ and $B(z,\sigma)= \{y\in\R^{2}:\psi(y)\neq 0\}$ we get that $\tilde{y}(t)\notin B(z,\sigma)$ for all $t>0$. In particular this proves that $\tilde{u}$ is an optimal solution of \eqref{Example:controlproblem:infinity}. Since $\tilde{y}(t)\notin B(z,\sigma)$ for all $t>0$, there exists $\tilde{t}>0$ such that $dist(\tilde{y}(\tilde{t}),\R\times\{0\})\geq\sigma$. On the other hand, we have that $\tilde{y}_n(t)e_2=0$ for all $t>0$, which gives the desired  contradiction.
\end{proof}

\begin{rem}
With the technique of the proof to the previous lemma it can also be argued that every  sequence of solutions $u^*_{\alpha_n}$ to  \eqref{Example:controlproblem} with initial condition $y_0 \in \R^2\setminus B(z,\sigma)$ and
 $\lim_{n\to \infty} \alpha_n=\infty$,  contains a  convergent subsequence  and every such subsequence  converges to a solution of \eqref{Example:controlproblem:infinity}.
\end{rem}.

\begin{theo}
    Let  $\bar{\alpha}>0$ and $\hat y_{\alpha,1}$ be as in \Cref{Example:lemma:nondif}. Then every optimal solution $u^{*}\in L^{2}((0,\infty);\R^{2})$ of \eqref{Example:controlproblem} with $y_0=(y_{0,1},0)$ satisfying $y_{0,1}<\hat y_{\alpha,1}$ is such $u^{*}(0)\cdot e_2\neq 0$  and \eqref{Example:controlproblem} admits at least two solutions. Moreover  for  $\alpha>\bar{\alpha}$ the value function $V_\alpha$ is not differentiable in $(y_1, 0)$ for each $ y_1 \in (-\infty,\hat y_{\alpha,1})$.
\end{theo}
\begin{proof}
    Let $\alpha>\bar{\alpha}$, $y_{0,1}<y_{\alpha,1}$ and $u^{*}\in L^{2}((0,\infty);\R^{2})$ be an optimal solution of \eqref{Example:controlproblem} with associated state $y^{*}\in H^{1}((0,\infty);\R^{2})$. By \eqref{Example:propOptimalityCond} we know that $u^{*},y^{*}\in C^{\infty}([0,\infty);\R^{2})$. By contradiction, let assume that $u^{*}(0)\cdot e_2= 0$. Then by \Cref{Example:solucionEje} we have  $y^{*}(t)\cdot  e_2=0$, for all $t>0$. Nevertheless, by \Cref{Example:lemma:nondif} there exists $\bar{t}>0$ such that $y^{*}(\bar{t})\cdot e_2\neq 0$ which is a contradiction. Hence  $u^{*}(0)\cdot e_2\neq 0$. Moreover, defining $\bar{u}\in L^{2}((0,\infty);\R^{2})$
    $$\bar{u}(t)=\left(\begin{array}{cc}
        1 & 0 \\
        0 & -1
    \end{array}\right)u^{*}(t),$$
    and $\bar{y}\in H^{1}((0,\infty);\R^{2})$ as the unique solution of $y'=\bar{u}$ with $y(0)=y_0$, we notice that $\ell_{\alpha}(\bar{y}(t))=\ell_{\alpha}(y^{*}(t))$ and $|\bar{u}(t)|^{2}=|u^{*}(t)|^{2}$. Consequently  $\bar{u}$ is an optimal solution of \eqref{Example:controlproblem} as well, different from $u^*$.

    If the value function $V_{\alpha}$ were differentiable at $y_0$ we would have that $D^{+} V_{\alpha}(y_0)$ would be a singleton. However, by \Cref{Example:NecessaryDynamiProgCond} we have that $-\beta\bar{u}(0)$ and $-\beta u^{*}(0)$ are contained in $ D^{+} V_{\alpha}(y_0)$ which is a contradiction to the fact that $\bar{u}(0)\neq u^{*}(0)$ and therefore the value function cannot be differentiable at $y_0$.
\end{proof}

\section{Concluding remarks}

In this work error bounds the convergence of feedback laws for infinite horizon problems with non-differentiable value functions were presented. These error bounds together with mollifications and the Moreau envelope permitted us to construct smooth sequences of feedback laws with the property that the associated value functions converge to the one of the original problem.
This result is important for both proving existence of solutions and convergence for the techniques developed in \cite{KuVaWal,KW1,KW2,KuVa,BokaWarinProst}, which train a feedback law by minimizing learning type cost functionals over a set of initial conditions. Further, \Cref{Corconv} says that the accumulation points of the controls obtained by this techniques are optimal controls for almost all initial condition.

As was seen, \Cref{Lyapunov} is important in order to ensure that the trajectories resulting from the feedback law approximation do not escape from the domain where the value function is approximated. Indeed, this hypothesis allowed us to estimate the escape time of the feedback from the approximation region depending on the regularity of the value function. In this respect, a stronger assumption on the stabilizability of the dynamics of \eqref{ControlProblem} could lead to strengthen the estimates in \Cref{Sec:StabilityEst}. For instance the, in the case of exponential stabilizability, the escape time of the trajectories resulting from the feedback law approximation is expected to be  infinity in a similar way as was done in \cite[Section 4, Proposition 1]{KuVaWal} for the smooth case.

The results of \Cref{Sec:ErrorEs} could also be applied to bound the error of feedback laws constructed from data driven approaches for the case of non-differentiable value function. However, in order to do this the constructed feedback laws must satisfy \eqref{ConvTheo2:eq0}, which of course will depend on the particularities of the method (e.g. neural network, polynomials, etc) and the control problem under study. Moreover, a stability assumptions like  \Cref{Lyapunov} needs to be satisfied by the closed loop problem obtained from the constructed feedback law. In this regards, it could be of interest to study under which conditions the results in  \Cref{Sec:ErrorEs} could be applied to prove the convergence of data driven approaches.

The example presented in section \Cref{Example} proves that even if all the data of the control problem are smooth,  the value function can be non-differentiable. On the other hand, for this example the value function was proved to be Lipschitz continuous and \Cref{Lyapunov} was also proved to hold. This underline the importance of the results presented in this work concerning the syntheses of smooth feedback laws. It is also worth to mention that  the non-uniqueness result is part due to the symmetry of both the objective function and the dynamics, which enable us to find a transformation of the optimal control which is still  optimal. Further, the particular characteristics of this problem made it  possible to demonstrate that the controls which are invariable under this transformation are not optimal solutions for $\alpha$ large enough. We expect that many interesting generalizations and modifications of this example are possible.

To conclude this work we mention some interesting extensions. We order them according to their apparent complexity and connection with this work. The error bound derived in \Cref{Sec:ErrorEs} and the escape time estimation in \Cref{Sec:StabilityEst} could be easily extend to the finite horizon case with some modification on the hypotheses. Additionally, the case of restrictions on the state could be treated by first approximating the problem by a penalty method and followed by the construction of a sequence of approximating feedback laws. Further, a key condition in this work was the availability of an expression of the feedback law as a function of the gradient of the value function. This is also the case for problems with convex restrictions on the control, where this expression is given by the projection onto the restrictions set. Nevertheless, since the projection on a closed convex set is just Lipschitz continuous, following the methods of the present work will lead to a Lipschitz sequence of feedback laws. Finally, since we proved error estimates in  \Cref{Sec:ErrorEs} with respect to $L_p$ norms with $p\in (1,\infty)$, we believe that it could be possible to apply similar techniques to analyze the case of control problems with discontinuous value functions, for instance problems with final cost and escape time.

\bibliographystyle{elsarticle-harv}
\bibliography{biblio}





\end{document}